\documentclass[11pt, a4paper,leqno]{amsart}
\usepackage{amsmath,amsthm,amscd,amssymb,amsfonts, amsbsy}
\usepackage{latexsym}
\usepackage{exscale}

\usepackage{pgf}

\calclayout
\allowdisplaybreaks

\newtheorem{theorem}{Theorem}[section]

\newtheorem{lemma}{Lemma}[section]

\theoremstyle{definition}
\newtheorem{definition}{Definition}
\newtheorem{remark}{Remark}[section]

\numberwithin{equation}{section}

\newcommand{\beq}{\begin{equation}}
\newcommand{\bea}[1]{\begin{array}{#1} }
\newcommand{\eeq}{ \end{equation}}
\newcommand{\ea}{ \end{array}}

\def \F  {{\mathscr {F}}}

\def \O {{\Omega}}

\def \d {{\delta}}

\def\mean#1{\mathchoice%
          {\mathop{\kern 0.2em\vrule width 0.6em height 0.69678ex depth -0.58065ex
                  \kern -0.8em \intop}\nolimits_{\kern -0.4em#1}}%
          {\mathop{\kern 0.1em\vrule width 0.5em height 0.69678ex depth -0.60387ex
                  \kern -0.6em \intop}\nolimits_{#1}}%
          {\mathop{\kern 0.1em\vrule width 0.5em height 0.69678ex
              depth -0.60387ex
                  \kern -0.6em \intop}\nolimits_{#1}}%
          {\mathop{\kern 0.1em\vrule width 0.5em height 0.69678ex depth -0.60387ex
                  \kern -0.6em \intop}\nolimits_{#1}}}

\def\vintslides_#1{\mathchoice%
          {\mathop{\kern 0.1em\vrule width 0.5em height 0.697ex depth -0.581ex
                  \kern -0.6em \intop}\nolimits_{\kern -0.4em#1}}%
          {\mathop{\kern 0.1em\vrule width 0.3em height 0.697ex depth -0.604ex
                  \kern -0.4em \intop}\nolimits_{#1}}%
          {\mathop{\kern 0.1em\vrule width 0.3em height 0.697ex depth -0.604ex
                  \kern -0.4em \intop}\nolimits_{#1}}%
          {\mathop{\kern 0.1em\vrule width 0.3em height 0.697ex depth -0.604ex
                  \kern -0.4em \intop}\nolimits_{#1}}}

\newcommand{\aveint}[2]{\mathchoice%
          {\mathop{\kern 0.2em\vrule width 0.6em height 0.69678ex depth -0.58065ex
                  \kern -0.8em \intop}\nolimits_{\kern -0.45em#1}^{#2}}%
          {\mathop{\kern 0.1em\vrule width 0.5em height 0.69678ex depth -0.60387ex
                  \kern -0.6em \intop}\nolimits_{#1}^{#2}}%
          {\mathop{\kern 0.1em\vrule width 0.5em height 0.69678ex depth -0.60387ex
                  \kern -0.6em \intop}\nolimits_{#1}^{#2}}%
          {\mathop{\kern 0.1em\vrule width 0.5em height 0.69678ex depth -0.60387ex
                  \kern -0.6em \intop}\nolimits_{#1}^{#2}}}

\def\eqn#1$$#2$${\begin{equation}\label#1#2\end{equation}}
\def\charfn_#1{{\raise1.2pt\hbox{$\chi
_{\kern-1pt\lower3pt\hbox{{$\scriptstyle#1$}}}$}}}
\def\diam{\operatorname{diam}}
\def\qq1{q_*}
\def\q2{q_{**}}
\def\dist{\operatorname{dist}}

\def\O{\rm O}

\newdimen\vintbar
\def\vint{-\kern-\vintbar\int}

\def\B{\mathcal B}

\def\F{\mathcal F}

\def\L{\mathcal L}
\def\K{\mathcal K}
\def\O{\mathcal O}
\def\P{\mathcal P}

\def\L{\mathcal L}

\def\W{\mathcal W}

\def\0{\boldsymbol 0}

\newcommand{\R}{\mathbb R}

\newtoks\by
\newtoks\paper
\newtoks\book
\newtoks\jour
\newtoks\yr
\newtoks\pages
\newtoks\vol
\newtoks\publ

\def\name[#1, #2]{#1 #2}
\def\ota{{\hbox{\bf ???}}}
\def\cLear{\by=\ota\paper=\ota\book=\ota\jour=\ota\yr=\ota
\pages=\ota\vol=\ota\publ=\ota}
\def\endpaper{\the\by, \textit{\the\paper},
{\the\jour} \textbf{\the\vol} (\the\yr), \the\pages.\cLear}
\def\endbook{\the\by, \textit{\the\book},
\the\publ, \the\yr.\cLear}
\def\endpap{\the\by, \textit{\the\paper}, \the\jour.\cLear}
\def\endproc{\the\by, \textit{\the\paper}, \the\book, \the\publ,
\the\yr, \the\pages.\cLear}

\renewcommand{\d}{\, \mathrm{d}} 

\begin{document}

\title[Estimates for parabolic measures]{On the fine properties of  parabolic measures\\ associated to strongly degenerate parabolic\\ operators of Kolmogorov type}


\address{Malte Litsg{\aa}rd \\Department of Mathematics, Uppsala University\\
S-751 06 Uppsala, Sweden}
\email{malte.litsgard@math.uu.se}

\address{Kaj Nystr\"{o}m\\Department of Mathematics, Uppsala University\\
S-751 06 Uppsala, Sweden}
\email{kaj.nystrom@math.uu.se}

\thanks{K. N was partially supported by grant  2017-03805 from the Swedish research council (VR)}

\author{Malte Litsg{\aa}rd and Kaj Nystr{\"o}m}
\maketitle
\begin{abstract}
\noindent \medskip
We consider strongly degenerate parabolic operators of the form
              \begin{eqnarray*}
   \L:=\nabla_X\cdot(A(X,Y,t)\nabla_X)+X\cdot\nabla_Y-\partial_t
    \end{eqnarray*}
    in unbounded domains
         \begin{eqnarray*}
 \Omega=\{(X,Y,t)=(x,x_{m},y,y_{m},t)\in\mathbb R^{m-1}\times\mathbb R\times\mathbb R^{m-1}\times\mathbb R\times\mathbb R\mid x_m>\psi(x,y,t)\}.
    \end{eqnarray*}
    We assume that $A=A(X,Y,t)$ is  bounded, measurable and uniformly elliptic (as a matrix in $\mathbb R^{m}$) and concerning $\psi$ and $\Omega$ we assume that $\Omega$ is what we call an (unbounded) Lipschitz domain: $\psi$ satisfies a
    uniform Lipschitz condition adapted to the dilation structure and the (non-Euclidean) Lie group underlying
    the operator $\L$. We prove, assuming in addition that $\psi$ is independent of the variable $y_m$, that $\psi$ satisfies an additional regularity condition formulated in terms of a Carleson measure, and additional conditions on $A$, that the associated parabolic measure is absolutely continuous with respect to a surface measure and that
    the associated Radon-Nikodym derivative defines an $A_\infty$-weight with respect to the surface measure.\\

\noindent
2000  {\em Mathematics Subject Classification.} 35K65, 35K70, 35H20, 35R03.
\noindent

\medskip

\noindent
{\it Keywords and phrases: Kolmogorov equation, parabolic, ultraparabolic, hypoelliptic, operators in divergence form,  Lipschitz domain, doubling measure, parabolic measure, Carleson measure, $A_\infty$, Lie group.}
\end{abstract}

    \setcounter{equation}{0} \setcounter{theorem}{0}
    \section{Background and motivation}
    In this paper we are concerned with the fine properties of parabolic measures, defined with respect to appropriate domains $\Omega$, and associated to  the operator
              \begin{eqnarray}\label{e-kolm-nd}
   \L=\L_A:=\nabla_X\cdot(A(X,Y,t)\nabla_X)+X\cdot\nabla_Y-\partial_t,
    \end{eqnarray}
    in $\mathbb R^{N+1}$, $N=2m$, $m\geq 1$, equipped with coordinates $(X,Y,t):=(x_1,...,x_{m},y_1,...,y_{m},t)\in \mathbb R^{m}\times\mathbb R^{m}\times\mathbb R$.  We assume that $A=A(X,Y,t)=\{a_{i,j}(X,Y,t)\}_{i,j=1}^{m}$ is a real-valued $m\times m$-dimensional symmetric matrix
    satisfying
    \begin{eqnarray}\label{eq2}
      \kappa^{-1}|\xi|^2\leq \sum_{i,j=1}^{m}a_{i,j}(X,Y,t)\xi_i\xi_j,\quad \ \ |A(X,Y,t)\xi\cdot\zeta|\leq \kappa|\xi||\zeta|,
    \end{eqnarray}
    for some $\kappa\in [1,\infty)$, and for all $\xi,\zeta\in \mathbb R^{m}$, $(X,Y,t)\in\mathbb R^{N+1}$. We refer to $\kappa$ as the  constant of $A$. Throughout the paper we will also assume that
        \begin{eqnarray}\label{eq2+}
    a_{i,j}\in C^\infty(\mathbb R^{N+1})
    \end{eqnarray}
    for all $i,j\in\{1,...,m\}$. While the  assumption in \eqref{eq2+} will only be used in a qualitative fashion, the constants of our quantitative estimates will depend on $m$ and $\kappa$.

    The starting point for our analysis is the recent results concerning  the local regularity of weak solutions to the equation $\L u=0$ established in \cite{Ietal}. In \cite{Ietal} the authors extended the  De Giorgi-Nash-Moser (DGNM) theory, which in its original form only considers elliptic
or parabolic equations in divergence form, to hypoelliptic equations with rough coefficients including the ones in \eqref{e-kolm-nd} assuming \eqref{eq2} and, implicitly, also \eqref{eq2+}. Their result is the correct scale- and translation-invariant estimates for local H{\"o}lder continuity and the Harnack inequality for weak solutions.

We recall that the prototype for the operators in \eqref{e-kolm-nd}, i.e. $A\equiv 1_m$ and the operator
    $$\K:=\nabla_X\cdot \nabla_X+X\cdot\nabla_Y-\partial_t,$$
 was originally introduced and studied by
Kolmogorov in a famous note published in 1934 in Annals of Mathematics, see \cite{K}. Kolmogorov noted that $\K$ is an example of a degenerate parabolic
operator having strong regularity properties and he proved that $\K$ has a fundamental solution which is smooth off its diagonal.
As a consequence,
\begin{eqnarray}\label{uu3}
  \K u = f \in C^\infty \quad \Rightarrow \quad u \in C^\infty,
\end{eqnarray}
for every distributional solution of $\K u=f$. These days, using the terminology introduced by H{\"o}rmander, see \cite{Hm}, the property in \eqref{uu3} is stated
\begin{eqnarray}\label{uu2}
\mbox{$\K$ is hypoelliptic}.
\end{eqnarray}
Naturally, for operators as in \eqref{e-kolm-nd}, assuming only measurable coefficients and \eqref{eq2}, the methods of Kolmogorov and H{\"o}rmander can not be directly applied to establish the DGNM theory and related estimates.

The results in \cite{Ietal} represent  an important achievement which paves the way for developments concerning
operators as in \eqref{e-kolm-nd} in several fields of analysis and in the theory of PDEs. In this paper we contribute to the understanding of the fine properties of the Dirichlet problems for operators of the form stated in \eqref{e-kolm-nd}
    in appropriate domains $\Omega\subset\mathbb R^{N+1}$ and we note that in general there is a rich interplay between the operators considered, applications and geometry. Indeed, today the Kolmogorov operator, and the more general operators of Kolmogorov-Fokker-Planck type with variable coefficients considered in this paper, play central roles in many application in analysis, physics and finance and depending on the application different model cases for the local geometry of $\Omega$ may be relevant:
    \begin{eqnarray}\label{dom-mod}
 (i)&&\{(X,Y,t)=(x,x_{m},y,y_{m},t)\in\mathbb R^{N+1}\mid x_m>\psi_1(x,Y,t)\},\notag\\
 (ii)&&\{(X,Y,t)=(x,x_{m},y,y_{m},t)\in\mathbb R^{N+1}\mid y_m>\psi_2(X,y,t)\},\\
 (iii)&&\{(X,Y,t)=(x,x_{m},y,y_{m},t)\in\mathbb R^{N+1}\mid t>\psi_3(X,Y)\}.\notag
\end{eqnarray}
In particular, in finance and in the context of option pricing and associated free boundary problems, case $(i)$ is relevant. In kinetic theory it is relevant to restrict the particles to a container making case $(ii)$ relevant. Case $(iii)$ captures, as a special case, the initial value or Cauchy problem.

In this paper we consider solutions to $\L u=0$ in  $\Omega$ assuming \eqref{eq2} and \eqref{eq2+}. Concerning $\Omega$ we restrict ourselves to case $(i)$ and unbounded domains $\Omega\subset\mathbb R^{N+1}$ of the form
\begin{eqnarray}\label{dom-}
 \Omega=\{(X,Y,t)=(x,x_{m},y,y_{m},t)\in\mathbb R^{N+1}\mid x_m>\psi(x,y,t)\}.
    \end{eqnarray}
    We impose restrictions on $\psi$ of Lipschitz character and the importance of the additional assumption that $\psi$ is independent of $y_m$ will be explained.

    Assuming that $\Omega\subset\mathbb R^{N+1}$ is a (unbounded) Lipschitz domain in the sense of
    Definition \ref{car} below, it follows that given $\varphi\in
C_0(\partial\Omega)$, there exists
a  unique  weak solution  $u=u_\varphi$, $u\in C(\bar \Omega)$, to the Dirichlet problem
\begin{equation} \label{e-bvpuu}
\begin{cases}
	\L u = 0  &\text{in} \ \Omega, \\
      u = \varphi  & \text{on} \ \partial \Omega.
\end{cases}
\end{equation}
Furthermore, there exists, for every $(Z, t):=(X,Y,t)\in \Omega$, a unique probability
measure  $\omega(Z,t,\cdot)$ on $\partial\Omega$ such that
\begin{eqnarray}  \label{1.1xxuu}
u(Z,t)=\iint_{\partial\Omega}\varphi(\tilde Z,\tilde t)\d \omega(Z,t,\tilde Z,\tilde t).
\end{eqnarray}
The measure $\omega(Z,t,E)$ is referred to as the parabolic measure associated to $\L$ in $\Omega$ and at $(Z, t)\in \Omega$ and of $E\subset\partial\Omega$. Properties of $\omega(Z,t,\cdot)$ govern the Dirichlet problem in \eqref{e-bvpuu}.

If  $\Omega=\Omega_\psi\subset\mathbb R^{N+1}$ is an unbounded ($y_m$-independent) Lipschitz domain we introduce the (physical) measure
  $\sigma$ on $\partial\Omega$ as
 \begin{eqnarray}\label{surfac+}d\sigma(X,Y,t):=\sqrt{1+|\nabla_{x}\psi(x,y,t)|^2}\d x\d Y\d t,\ (X,Y,t)\in\partial\Omega.
 \end{eqnarray}
We will refer to  $\sigma$ as the surface measure on $\partial\Omega$.

Two fundamental questions concerning $\omega(Z,t,\cdot)$ can be stated as follows. Under what assumptions on $A$  and $\psi$, $\Omega$ as in \eqref{dom-}, is it true that
\begin{align}  \label{problems}
(i)&\mbox{ $\omega(Z,t,\cdot)$ is a doubling measure ?}\notag\\
(ii)&\mbox{ $\omega(Z,t,\cdot)$ satisfies scale-invariant absolute continuity estimates with respect}\\
&\mbox{ to the physical (surface) measure $\sigma$ on $\partial\Omega$ ?}\notag
\end{align}

In \cite{LN} we developed a potential theory for operators $\L$ as in \eqref{e-kolm-nd}, assuming only \eqref{eq2} and \eqref{eq2+}, in unbounded ($y_m$-independent) Lipschitz domains in the sense of Definition \ref{car} below. As part of this theory we proved that $\omega(Z,t,\cdot)$ is a doubling measure, hence establishing \eqref{problems} $(i)$. The additional assumption that the function $\psi$ defining the domain is independent of $y_m$ was cruical in this part of \cite{LN}. In this paper we refine the result of \cite{LN} considerably by proving, under additional assumptions on $\psi$ ( i.e. on $\Omega$) and the coefficients $A$, that $\omega(Z,t,\cdot)$ defines an $A_\infty$ weight with respect to the surface measure $\sigma$ in \eqref{surfac+} giving a quantitative answer to \eqref{problems} $(ii)$.  In the prototype case $A\equiv 1_m$, i.e. in the case of the operator $\K$, the corresponding results were established in \cite{NP} and \cite{N1}, respectively, and this seems to be the only previous results of their kind for operators of Kolmogorov type.

To put the  results of \cite{LN} and this paper into perspective it is relevant to outline the progress on the corresponding problems in the case of uniformly parabolic equations in $\mathbb R^{m+1}$, i.e. in the case when all dependence on the variable $Y$ is removed in \eqref{e-kolm-nd} leaving us with the operator
              \begin{eqnarray}
\nabla_X\cdot(A(X,t)\nabla_X)-\partial_t.
    \end{eqnarray}
    In this setting  the questions in \eqref{problems} have in recent times been discussed and resolved in a number of fundamental papers and  we here highlight the main contributions to the field.

First,  for uniformly parabolic equations with bounded measurable coefficients in Lipschitz type domains, scale and translation invariant boundary comparison principles,
boundary Harnack inequalities and doubling properties of associated parabolic measures  were settled in a number of fundamental papers including \cite{FS}, \cite{FSY}, \cite{SY}, \cite{FGS} and \cite{N}. This type of results find their applications in many fields of analysis including the analysis of free boundary problems, see \cite{C1}, \cite{C2} and \cite{ACS} for instance.

Second, in \cite{LS}, \cite{LM}, \cite{HL}, \cite{H}, see also \cite{HL1}, the correct notion
    of time-dependent Lipschitz type cylinders, correct from the perspective of  parabolic measure, parabolic singular integral operators, parabolic
    layer potentials, as well as from the perspective of the Dirichlet, Neumann and Regularity
problems with data in $L^p$ for the heat operator, was found.  In particular, in \cite{LS}, \cite{LM} the mutual absolute continuity of the
parabolic measure with respect to surface measure, and the $A_\infty$-property, was studied/established and in \cite{HL} the authors
solved the Dirichlet, Neumann and Regularity
problems with data in $L^2$.  For further related results concerning the fine
properties of parabolic measures  we refer to the impressive and influential work \cite{HL2}. In \cite{HL2} the authors consider equations modeled on certain refined pull-backs of the heat operator to the  parabolic upper half space $\mathbb R^{m+1}_+=\{(x,x_m,t)\mid x_m>0\}$. These  pull-back operators take the form
\begin{eqnarray}\label{e-kolm-ndfl+a}
              \nabla_X\cdot(A\nabla_Xu)+B\nabla_Xu-\partial_tu=0,
    \end{eqnarray}
    where  the coefficient $B$ now gives rise to a singular drift term  and the regularity of $A$ and $B$ are measured using certain Carleson measures. The singular drift term complicates matters considerably as there seem to be no positive answer to \eqref{problems} $(i)$ in this case. It should be mentioned that in
    \cite{NR}, \cite{DPP}, parts of \cite{HL2} have been simplified.

    Third, very recently there has been significant progress in the theory of boundary value problems for second order parabolic equations (and systems) of the form
    \begin{eqnarray}\label{eq1}
\nabla_X\cdot(A(x,t)\nabla_Xu)-\partial_tu=0,
    \end{eqnarray}
    in the   parabolic upper half space $\mathbb R_+^{m+1}$ with boundary determined by $x_m=0$, assuming only bounded, measurable, uniformly elliptic and complex coefficients. In~\cite{N2, CNS, N3}, the solvability for Dirichlet, Regularity and Neumann problems with data in $L^2$ were established for the class of parabolic equations \eqref{eq1} under the additional assumptions that the elliptic part is also independent of the time variable $t$ and that it has either constant (complex) coefficients, real symmetric coefficients, or small perturbations thereof. Focusing on parabolic measure, a particular consequence of Theorem 1.3 in~\cite{CNS} is the generalization of~\cite{FSa} to equations of the form \eqref{eq1} but with $A$ real, symmetric and time-independent. This analysis in \cite{N2, CNS, N3} was advanced further in~\cite{AEN}, where a first order strategy to study boundary value problems of parabolic systems with second order elliptic part in the upper half-space was developed. The outcome of~\cite{AEN} was the possibility to address arbitrary parabolic equations (and systems) as in \eqref{eq1} with coefficients depending also on time and on the transverse variable with additional transversal regularity. Finally, in \cite{AEN1} the authors consider parabolic equations as in \eqref{eq1}, assuming that the coefficients are real, bounded, measurable, uniformly elliptic, but not necessarily symmetric. They prove that the associated parabolic measure is  absolutely continuous with respect to the surface measure on
    $\mathbb R^{m+1}$ (i.e. $\mathrm{d} x\d t$) in the sense defined by the Muckenhoupt class $A_\infty(\mathrm{d} x\d t)$.

    In light of the above outline concerning the progress on uniformly parabolic equations, \cite {LN} and the main result of this paper,  Theorem \ref{Ainfty} stated below, represent important steps towards a corresponding theory concerning the Dirichlet problem for operators of Kolmogorov type with bounded and measurable coefficients in Lipschitz type domains adapted to the (non-Euclidean) group structure.

    The rest of the paper is organized as follows. Section \ref{sec2} is of preliminary nature. In Section \ref{sec3} we state our main result, Theorem \ref{Ainfty}, the proof of which we start in
    Section \ref{sec4}. In Section \ref{sec4} we prove how Theorem \ref{Ainfty} can be reduced to three lemmas: Lemmas \ref{existcover}-\ref{lemmacruc+}. We consider the proof of Lemma \ref{lemmacruc+} a rather  difficult part in the proof of Theorem \ref{Ainfty} and in Section \ref{sec4} we show that this lemma can be reduced to one key lemma: Lemma \ref{Carleson}. Section \ref{sec5} is devoted to the proof of Lemma \ref{Carleson}.  Finally, in Section \ref{sec6} we prove Lemma \ref{existcover} and Lemma \ref{lemmacruc} by partially relying on a number of estimates for non-negative solutions recently established in \cite{LN}.

     \setcounter{equation}{0} \setcounter{theorem}{0}
    \section{Preliminaries}\label{sec2}

\subsection{Group law and metric}  The natural family of dilations for $\L$, $(\delta_r)_{r>0}$, on $\R^{N+1}$,
is defined by
\begin{equation}\label{dil.alpha.i}
 \delta_r (X,Y,t) =(r X, r^3 Y,r^2 t),
\end{equation}
for $(X,Y,t) \in \R^{N +1}$,  $r>0$.  Our class of operators  is closed under the  group law
\begin{equation}\label{e70}
 (\tilde Z,\tilde t)\circ (Z,t)=(\tilde X,\tilde Y,\tilde t)\circ (X, Y,t)=(\tilde X+X,\tilde Y+Y-t\tilde X,\tilde t+t), 
\end{equation}
where $(Z,t),\ (\tilde Z,\tilde t)\in \R^{N+1}$. Note that
\begin{equation}\label{e70+}
(Z,t)^{-1}=(X,Y,t)^{-1}=(-X,-Y-tX,-t),
\end{equation}
and hence
\begin{equation}\label{e70++}
 (\tilde Z,\tilde t)^{-1}\circ (Z,t)=(\tilde  X,\tilde  Y,\tilde  t)^{-1}\circ (X,Y,t)=(X-\tilde  X,Y-\tilde  Y+(
t-\tilde  t)\tilde  X,t-\tilde  t),
\end{equation}
whenever $(Z,t),\ (\tilde Z,\tilde t)\in \R^{N+1}$.  Given $(Z,t)=(X,Y,t)\in \R^{N+1}$ we let \begin{equation}\label{kolnormint}
  \|(Z, t)\|= \|(X,Y, t)\|:=|(X,Y)|\!+|t|^{\frac{1}{2}},\ |(X,Y)|=\big|X\big|+\big|Y\big|^{1/3}.
\end{equation}
We recall the following pseudo-triangular
inequality: there exists a positive constant ${c}$ such that
\begin{eqnarray}\label{e-ps.tr.in}
 \|(Z,t)^{-1}\|\le {c}  \| (Z,t) \|,\quad \|(Z,t)\circ (\tilde Z,\tilde t)\| \le  {c}  (\| (Z,t) \| + \| (\tilde Z,\tilde t)
\|),
\end{eqnarray}
whenever $(Z,t),(\tilde Z,\tilde t)\in \R^{N+1}$. Using \eqref{e-ps.tr.in} it  follows directly that
\begin{equation} \label{e-triangularap}
    \|(\tilde Z,\tilde t)^{-1}\circ (Z,t)\|\le c \, \|(Z,t)^{-1}\circ (\tilde Z,\tilde t)\|,
\end{equation}
whenever $(Z,t),(\tilde Z,\tilde t)\in \R^{N+1}$. Let
\begin{equation}\label{e-ps.distint}
    d((Z,t),(\tilde Z,\tilde t)):=\frac 1 2\bigl( \|(\tilde Z,\tilde t)^{-1}\circ (Z,t)\|+\|(Z,t)^{-1}\circ (\tilde Z,\tilde t)\|).
\end{equation}
Using \eqref{e-triangularap} it follows that
\begin{equation}\label{e-ps.dist}
   \|(\tilde Z,\tilde t)^{-1}\circ (Z,t)\|\approx d((Z,t),(\tilde Z,\tilde t))\approx \|(Z,t)^{-1}\circ (\tilde Z,\tilde t)\|
\end{equation}
for all $(Z,t),(\tilde Z,\tilde t)\in \R^{N+1}$ and with uniform constants. Again using \eqref{e-ps.tr.in} we also see that
\begin{equation} \label{e-triangular}
    d((Z,t),(\tilde Z,\tilde t))\le {c} \bigl(d((Z,t),(\hat Z,\hat t))+d((\hat Z,\hat
t),(\tilde Z,\tilde t))\bigr ),
\end{equation}
whenever $(Z,t),(\hat Z,\hat t),(\tilde Z,\tilde t)\in \R^{N+1}$, and hence $d$ is a symmetric quasi-distance. Based on $d$ we introduce the balls
\begin{equation}\label{e-BKint}
    \mathcal{B}_r(Z,t):= \{ (\tilde Z,\tilde t) \in\mathbb R^{N+1} \mid d((\tilde Z,\tilde t),(Z,t)) <
r\},
\end{equation}
for $(Z,t)\in \R^{N+1}$ and $r>0$.  The measure of the ball $\mathcal{B}_r(Z,t)$ is  $|\mathcal{B}_r(Z,t)|\approx r^{{\bf q}}$, independent of $(Z,t)$, and where
$${\bf q}:=4m+2.$$
Similarly, given $(z,t)=(x,y,t)\in \mathbb R^{N-1}=\mathbb R^{m-1}\times\mathbb R^{m-1}\times\mathbb R$ we let
\begin{equation}\label{e-BKint+}
    \mathcal{B}_r(z,t):= \{ (\tilde z,\tilde t) \in\mathbb R^{N-1} \mid d((\tilde x,0,\tilde y, 0,\tilde t),(x,0,y,0,t)) <
r\}.
\end{equation}
The measure of the ball $\mathcal{B}_r(z,t)$ is $|\mathcal{B}_r(z,t)|\approx r^{{\bf q}-4}$, independent of $(z,t)$. We will by $\mathcal{B}_r(Z,t)$  always denote a ball in $\R^{N+1}$, with capital $Z$, and by $\mathcal{B}_r(z,t)$ we will always denote a ball in $\R^{N-1}$, with lowercase $z$.

\subsection{Geometry}  We consider domains of the form stated in \eqref{dom-} and we here outline the assumptions we impose on the defining function $\psi$. Let
    $\P\in C_0^\infty(\mathcal{B}_1(0,0))$, $\mathcal{B}_1(0,0)\subset\mathbb R^{N-1}$, be a standard approximation of the identity. Let
    $$\P_\lambda(x,y,t)=\lambda^{-{(\bf q}-4)}\P(\lambda^{-1}x,\lambda^{-3}y,\lambda^{-2}t),$$ for $\lambda>0$.  Given a function $f$ defined on
$\mathbb R^{N-1}$ we let
\begin{eqnarray}\label{eq1vi}
\P_{\lambda}f(x,y, t)&:=&\iint_{\R^{N-1}}f(\bar x,\bar  y,\bar t)\P_\lambda((\bar  x,\bar y, \bar t)^{-1}\circ (x,y, t))\, \d\bar  x \d\bar  y \d\bar  t\notag\\
&=&\iint_{\R^{N-1}}f(\bar x,\bar  y,\bar t)\P_\lambda(x-\bar x,y-\bar y+(t-\bar t)\bar x,t-\bar t)\, \d\bar  x \d\bar  y \d\bar  t.
\end{eqnarray}
$\P_{\lambda}f$ represents a regularization of $f$. Given $(\tilde z, \tilde t)\in \R^{N-1}$, $\lambda>0$, we let $\gamma_\psi(\tilde z, \tilde t,\lambda)$ denote the number
\begin{eqnarray*}\label{eq1apa}
\biggl (\lambda^{-{({\bf q}-4)}}\iint_{\mathcal{B}_\lambda(\tilde z, \tilde t)}\biggl |\frac {\psi(\bar x,\bar y,\bar t)-\psi(\tilde x,\tilde y,\tilde t)-\P_\lambda(\nabla_x\psi)(\tilde x,\tilde y,\tilde t)(\bar x-\tilde x)}{\lambda}\biggr |^2\, \d\bar x\d\bar y\d\bar t\biggr )^{1/2}.
\end{eqnarray*}

   \begin{definition}\label{car} Assume that there exist constants $0<M_1,M_2<\infty$, such that
         \begin{eqnarray}\label{Lip-++a1}
|\psi(z,t)-\psi(\tilde z,\tilde t)|\leq M_1||(\tilde z,\tilde t)^{-1}\circ(z,t)||,
\end{eqnarray}
whenever $(z,t),\ (\tilde z,\tilde t)\in\mathbb R^{N-1}$ and such that
      \begin{eqnarray}\label{Lip-++a2}
\sup_{(z,t)\in\mathbb R^{N-1},\ r>0}\quad r^{-{({\bf q}-4)}}\int_0^r\iint_{\mathcal{B}_\lambda(z,t)}\gamma_\psi^2(\tilde z, \tilde t,\lambda)\, \frac {\d\tilde z\d\tilde t\d\lambda}\lambda\leq M_2.
\end{eqnarray}
Let $\Omega=\Omega_\psi$ be defined as in \eqref{dom-}. We say that $\Omega$, defined by a function $\psi$ satisfying \eqref{Lip-++a1}, is
 a ($y_m$-independent) Lipschitz domain  with constant $M_1$. We say that $\Omega$, defined by a function $\psi$ satisfying \eqref{Lip-++a1} and \eqref{Lip-++a2}, is an admissible ($y_m$-independent) Lipschitz domain  with constants $(M_1,M_2)$.
\end{definition}

Given $\rho>0$ and $\Lambda>0$ we introduce points of reference in $\mathbb R^{m-1}\times\mathbb R\times\mathbb R^{m-1}\times\mathbb
R\times\mathbb R$,
\begin{align}\label{pointsref2}
A_{\rho,\Lambda}^\pm:= \left(0,\Lambda\rho,0,\mp\tfrac 2 3\Lambda\rho^3,\pm\rho^2\right),\ A_{\rho,\Lambda}&:=\left(0,\Lambda\rho,0,0,0\right).
\end{align}
Given $(Z_0,t_0)\in\mathbb R^{N+1}$ we let $$A_{\rho,\Lambda}^\pm(Z_0,t_0):=(Z_0,t_0)\circ A_{\rho,\Lambda}^\pm,\ A_{\rho,\Lambda}(Z_0,t_0):=(Z_0,t_0)\circ A_{\rho,\Lambda}.$$

\subsection{Dyadic grids, Whitney cubes and Carleson boxes}\label{dya}     Assuming that $\Omega=\Omega_\psi\subset\mathbb R^{N+1}$ is a Lipschitz domain,  with constant $M_1$, in the sense of Definition \ref{car}, we let
$$\Sigma:=\partial \Omega=\{(x,x_{m},y,y_{m},t)\in\mathbb R^{N+1} \mid x_m=\psi(x,y,t)\}.$$
Then  $(\Sigma,d,d\sigma)$, where the symmetric quasi-distance $d$ was introduced in \eqref{e-ps.distint}, is a space of homogeneous type in the sense of \cite{CW} with homogeneous dimension  ${\bf q}-1$. Furthermore, $(\mathbb R^{N+1},d,dZdt)$ is also a space of homogeneous type in the sense of \cite{CW}, but with homogeneous dimension  ${\bf q}$. By the results in \cite{Ch} there exists what we here will refer to as a dyadic grid on
 $\Sigma$ having a number of important properties in relation to $d$. To formulate this we introduce, for any $(Z,t)=(X,Y,t)\in\Sigma$ and
$E\subset \Sigma$,
\begin{equation}
   {\rm dist} ((Z,t),E):=\inf \{ d((Z,t),(\tilde Z,\tilde t)) \mid (\tilde Z,\tilde t)\in E\},
\end{equation}
and we let
\begin{equation}
    \diam(E):=\sup \{ d((Z,t),(\tilde Z,\tilde t)) \mid (Z,t),\ (\tilde Z,\tilde t)\in E\}.
\end{equation}
Using  \cite{Ch} we can conclude that there exist
constants $ \alpha>0,\, \beta>0$ and $c_*<\infty$,  such that for each $k \in \mathbb{Z}$
there exists a collection of Borel sets, $\mathbb{D}_k$,  which we will call cubes, such that
$$
\mathbb{D}_k:=\{Q_{j}^k\subset\Sigma \mid j\in \mathfrak{I}_k\},$$ where
$\mathfrak{I}_k$ denotes some  index set depending on $k$, satisfying
\begin{eqnarray}\label{cubes}
(i)&&\mbox{$\Sigma=\cup_{j}Q_{j}^k\,\,$ for each
$k\in{\mathbb Z}$.}\notag\\
(ii)&&\mbox{If $m\geq k$ then either $Q_{i}^{m}\subset Q_{j}^{k}$ or
$Q_{i}^{m}\cap Q_{j}^{k}=\emptyset$.}\notag\\
(iii)&&\mbox{For each $(j,k)$ and each $m<k$, there is a unique
$i$ such that $Q_{j}^k\subset Q_{i}^m$.}\notag\\
(iv)&&\mbox{$\diam\big(Q_{j}^k\big)\leq c_* 2^{-k}$.}\notag\\
(v)&&\mbox{Each $Q_{j}^k$ contains $\Sigma\cap \mathcal{B}_{\alpha2^{-k}}(Z^k_{j},t^k_{j})$ for some $(Z^k_{j},t^k_j)\in\Sigma$.}\notag\\
(vi)&&\mbox{$\sigma(\{(Z,t)\in Q^k_j\mid{\rm dist}((Z,t),\Sigma\setminus Q^k_j)\leq \rho \,2^{-k}\big\})\leq
c_*\,\rho^\beta\,\sigma(Q^k_j),$}\notag\\
&&\mbox{for all $k,j$ and for all $\rho\in (0,\alpha)$.}
\end{eqnarray}
In the setting of a general space of homogeneous type, this result is due to Christ
\cite{Ch}, with the
dyadic parameter $1/2$ replaced by some constant $\delta \in (0,1)$. In fact, one may always take $\delta = 1/2$, see \cite[Proof of Proposition 2.12]{HMMM}.  We shall denote by  $\mathbb{D}=\mathbb{D}(\Sigma)$ the collection of all
$Q^k_j$, i.e. $$\mathbb{D} := \cup_{k} \mathbb{D}_k.$$
Note that \eqref{cubes} $(iv)$ and $(v)$ imply that for each cube $Q\in\mathbb{D}_k$,
there is a point $(Z_Q,t_Q)=(X_Q,Y_Q,t_Q)\in \Sigma$, and  a ball $\mathcal{B}_{r}(Z_Q,t_Q)$ such that
$r\approx 2^{-k} \approx {\rm diam}(Q)$
and \begin{equation}\label{cube-ball}
\Sigma\cap\mathcal{B}_{r}(Z_Q,t_Q)\subset Q \subset \Sigma\cap \mathcal{B}_{cr}(Z_Q,t_Q),\end{equation}
for some uniform constant $c$. We will denote the associated surface ball by
\begin{equation}\label{cube-ball2}
\Delta_Q:= \Sigma\cap \mathcal{B}_{r}(Z_Q,t_Q)\end{equation}
and we shall refer to the point $(Z_Q,t_Q)$ as the center of $Q$. Given a dyadic cube $Q\subset\Sigma$, we define its $\gamma$ dilate  by
\begin{equation}\label{dilatecube}
\gamma Q:= \Sigma\cap  \mathcal{B}_{\gamma \diam(Q)}(Z_Q,t_Q).
\end{equation}
For a dyadic cube $Q\in \mathbb{D}_k$, we let $\ell(Q) = 2^{-k}$, and we shall refer to this quantity as the length
of $Q$.  Clearly, $\ell(Q)\approx \diam(Q).$ For a dyadic cube $Q \in \mathbb{D}$, we let $k(Q)$ denote the dyadic generation
to which $Q$ belongs, i.e. we set  $k = k(Q)$ if
$Q\in \mathbb{D}_k$, thus, $\ell(Q) =2^{-k(Q)}$.  For any $Q\in \mathbb D(\Sigma)$, we set $\mathbb D_Q:= \{Q'\in\mathbb D \mid Q'\subset Q\}\,.$

Using that also $(\mathbb R^{N+1},d,dZdt)$ is a space of homogeneous type we see that we can partition
$\Omega$ into a collection
of (closed) dyadic Whitney cubes $\{I\}$, in the following denoted   $\mathcal{W}=\W(\Omega)$, such that the cubes in $\mathcal{W}$
form a covering of $\Omega$ with non-overlapping interiors, and
\begin{equation}\label{eqWh1} 4\, {\rm{diam}}\,(I)\leq \dist(4 I,\Sigma) \leq  \dist(I,\Sigma) \leq 40 \, {\rm{diam}}\,(I)\end{equation}
and
\begin{equation}\label{eqWh2}\diam(I_1)\approx \diam(I_2), \mbox{ whenever $I_1$ and $I_2$ touch.}
\end{equation}
Given $I\in \mathcal{W}$ we let $\ell(I)$ denote its size. Given $Q\in \mathbb D(\Sigma)$  we set
\begin{equation}\label{eq2.1}
\W_Q:= \left\{I\in \W\mid \,100^{-1} \ell(Q)\leq \ell(I)
\leq 100\,\ell(Q),\, {\rm and}\, \dist(I,Q)\leq 100\, \ell(Q)\right\}.
\end{equation}
We fix a small, positive parameter $\tau$, and given $I\in\W$,
we let
\begin{equation}\label{eq2.3*}I^* =I^*(\tau) := (1+\tau)I
\end{equation}
denote the corresponding ``fattened" Whitney cube. Choosing $\tau$ small we see that the cubes $I^*$ will retain the usual properties of Whitney cubes;
in particular, that
$$\diam(I) \approx \diam(I^*) \approx \dist(I^*,\Sigma) \approx \dist(I,\Sigma)\,.$$
We then define a  Whitney region
with respect to $Q$ by setting
\begin{equation}\label{eq2.3}
U_Q:= \bigcup_{I\in \W_Q}I^*\,. 
\end{equation}
Given $Q\in \mathbb D(\Sigma)$  we let
\begin{equation}\label{eq2.box-}
T_Q:={\rm int}\left( \bigcup_{Q'\in \mathbb D_Q} U_{Q'}\right),
\end{equation}
denote  the Carleson box
associated to $Q$. Furthermore, given $\gamma\geq 1$ we let
\begin{equation}\label{eq2.box}
T_{\gamma Q}:={\rm int}\left( \bigcup_{Q':\ Q'\cap (\gamma Q)\neq \emptyset} U_{Q'}\right),
\end{equation}
denote  the Carleson set
associated to the $\gamma$ dilate of $Q$. Finally, given $Q\in\mathbb{D}$ and $\Lambda>0$, we let
\begin{equation}\label{pointsref2apa}
\begin{split}
A_{Q,\Lambda}^\pm&:=(Z_Q,t_Q)\circ (0,\Lambda l(Q),0,\mp\frac 2 3\Lambda l(Q)^3,\pm l(Q)^2),\\
A_{Q,\Lambda}&:=(Z_Q,t_Q)\circ (0,\Lambda l(Q),0,0,0).
\end{split}
\end{equation}

\subsection{Weak solutions} Consider $U_X\times U_Y\times J\subset\mathbb R^{N+1}$ with $U_X\subset\mathbb R^{m}$, $U_Y\subset\mathbb R^{m}$ being bounded domains, i.e, open, connected and bounded sets, and $J=(a,b)$ with $-\infty<a<b<\infty$. Then
    $u$ is said to be a weak solution to the equation
                  \begin{eqnarray}\label{e-kolm-nd-}
   \L u=\nabla_X\cdot(A(X,Y,t)\nabla_Xu)+X\cdot\nabla_Yu-\partial_tu=0,
    \end{eqnarray}
    in $U_X\times U_Y\times J\subset\mathbb R^{N+1}$ if
                      \begin{eqnarray}\label{weak1}
                      u\in L_{Y,t}^2(U_Y\times J,H_X^1(U_X)),
    \end{eqnarray}
    and
                          \begin{eqnarray}\label{weak2}
                          -X\cdot\nabla_Yu+\partial_tu\in  L_{Y,t}^2(U_Y\times J,H_X^{-1}(U_X)),
    \end{eqnarray}
    and if $\L u=0$ in the sense of distributions, i.e,
                              \begin{eqnarray}\label{weak3}
                              \iiint_{}\ \bigl(A(X,Y,t)\nabla_Xu\cdot \nabla_X\phi+(X\cdot \nabla_Y\phi)u-u\partial_t\phi\bigr )\, \d X \d Y \d t=0,
                               \end{eqnarray}
                               whenever $\phi\in C_0^\infty(U_X\times U_Y\times J)$.

We say that $u$ is a weak solution
to the equation $\L u=0$ in $\Omega$ if $u$ is a weak solution to $\L u=0$ in $U_X\times U_Y\times J\subset\mathbb R^{N+1}$, where $U_X\subset\mathbb R^{m}$, $U_Y\subset\mathbb R^{m}$ are  bounded domains, and $J=(a,b)$ with $-\infty<a<b<\infty$, whenever $U_X\times U_Y\times J$ is compactly contained in $\Omega$.

        \setcounter{equation}{0} \setcounter{theorem}{0}
    \section{Statement of the main result}\label{sec3}
Assume that $\Omega=\Omega_\psi\subset\mathbb R^{N+1}$ is a Lipschitz domain,  with constant $M_1$, in the sense of Definition \ref{car}, and recall that $\mathbb{D}$ is the set of dyadic cubes on $\partial\Omega$. Given $Q\in\mathbb{D}$, recall the definitions of $l(Q)$, $(Z_Q,t_Q)$, $\gamma Q$, $T_Q$, $A_{Q,\Lambda}^\pm$, introduced in Subsection \ref{dya}.

Using this notation a version of one of the main results (namely, Theorem 3.6) proved in \cite{LN} can be stated as follows.

\begin{theorem}\label{dub} Let $\Omega\subset\mathbb R^{N+1}$ is an unbounded ($y_m$-independent) Lipschitz domain with constant $M_1$ in the sense of Definition \ref{car}. Assume that  $A$ satisfies \eqref{eq2} with constant $\kappa$, \eqref{eq2+} and that
\begin{eqnarray}\label{struct}
A(X,Y,t)=A(x,x_m,y,y_m,t)=A(x,x_m,y,t)
\end{eqnarray}
whenever $(x,x_{m},y,y_{m},t)\in\mathbb R^{N+1}$, i.e. also $A$ is assumed to be independent of the variable $y_m$.  Then there exist
 $\Lambda=\Lambda(m,M_1)$, $1\leq \Lambda<\infty$,   $c=c(m,\kappa,M_1)$,  $1\leq c<\infty$, such that the following is true. Consider $Q_0\in\mathbb{D}$ and let $\omega(\cdot):=\omega\bigl (A_{cQ_0,\Lambda}^+,\cdot\bigr )$. Then
\begin{eqnarray*}
\omega\bigl (2Q\bigr )\leq
c\omega\bigl (Q\bigr )
\end{eqnarray*}
for all $Q\in\mathbb{D}$ such that $4Q\subset Q_0$.
\end{theorem}

Given an  unbounded ($y_m$-independent)  Lipschitz domain $\Omega=\Omega_\psi\subset\mathbb R^{N+1}$ we let $\delta=\delta(X,Y,t)$ denote the distance from $(X,Y,t)\in\Omega$ to $\partial\Omega$, i.e.
\begin{equation}\label{deltadist}
    \delta(X,Y,t)=\min\{d((X,Y,t),(\tilde X,\tilde Y,\tilde t)) \mid (\tilde X,\tilde Y,\tilde t)\in\partial\Omega\}.
\end{equation}
Consider the following measures $\mu_1$ and $\mu_2$ defined on $\Omega$:
\begin{equation}\label{measure1}
\begin{split}
\d\mu_1(X,Y,t)&:=|\nabla_XA(X,Y,t)|^2\delta(X,Y,t)\ \d X\d Y\d t,\\
\d\mu_2(X,Y,t)&:=|(X\cdot\nabla_Y-\partial_t)A(X,Y,t)|^2\delta^3(X,Y,t)\ \d X\d Y\d t.
\end{split}
\end{equation}
We say that $\mu_1$ and $\mu_2$ are Carleson measures on $\Omega$ with constant $\Gamma$ if
\begin{equation}\label{measure2}
\begin{split}
\sup_{Q\in\mathbb{D}}\quad l(Q)^{-{({\bf q}-1)}}\iiint_{T_Q}\d\mu_1(\tilde X,\tilde Y,\tilde t)&\leq\Gamma,\\
\sup_{Q\in\mathbb{D}}\quad l(Q)^{-{({\bf q}-1)}}\iiint_{T_Q}\d\mu_2(\tilde X,\tilde Y,\tilde t)&\leq\Gamma.
\end{split}
\end{equation}

The following is the main result proved in this paper.
\begin{theorem}\label{Ainfty} Assume that $\Omega\subset\mathbb R^{N+1}$ is an (unbounded) admissible ($y_m$-independent) Lipschitz domain  with constants $(M_1,M_2)$ in the sense of Definition \ref{car}. Assume that $A$ satisfies \eqref{eq2} with constant $\kappa$, \eqref{eq2+} and \eqref{struct}, i.e. also $A$ is independent of $y_m$. Assume that the measures $\mu_1$ and $\mu_2$ defined in \eqref{measure1} are Carleson measures on $\Omega$ with constant $\Gamma$ in the sense of \eqref{measure2}. Then there exist
 $\Lambda=\Lambda(m,M_1)$, $1\leq \Lambda<\infty$,   $c=c(m,\kappa,M_1)$,  $1\leq c<\infty$, $\tilde c=\tilde c(m,\kappa, M_1,M_2,\Gamma)$,  $1\leq \tilde c<\infty$,
   $\eta=\eta(m,\kappa,M_1,M_2,\Gamma)$, $0<\eta<1$, such that the following is true. Consider $Q_0\in\mathbb{D}$ and let $\omega(\cdot):=\omega\bigl (A_{cQ_0,\Lambda}^+,\cdot\bigr )$. Then
 \begin{eqnarray*}
 \quad\tilde c^{-1}\biggl (\frac{ \sigma ( E ) }{ \sigma(Q)}\biggr )^{1/\eta}\leq \frac {\omega\bigl (E\bigr )}{\omega\bigl ( Q\bigr )}\leq \tilde c\biggl (\frac{ \sigma ( E ) }{ \sigma(Q)}\biggr )^\eta
\end{eqnarray*}
whenever $E\subset Q$ for some $Q\in\mathbb{D}$ such that $Q\subseteq Q_0$.
\end{theorem}

As mentioned before,  in the prototype case $A\equiv 1_m$, i.e. in the case of the operator $\K$, Theorem \ref{Ainfty} is proved in \cite{N} and this seems to be the only previous result of its kind for operators of Kolmogorov type.

    \section{Proof of Theorem \ref{Ainfty}: preliminary reductions}\label{sec4}

     Using Lemma \ref{lemmacruc-} and Lemma \ref{T:doubling} below  it follows that it suffices to prove Theorem \ref{Ainfty} with $Q=Q_0$. In the following we let $Q_0\in \mathbb{D}$ and we let $\omega(\cdot)$ be as in the statement of Theorem \ref{Ainfty}.  Our proof of  Theorem \ref{Ainfty} is based on ideas introduced in \cite{KKPT} in the context of elliptic measures and we will use  the notion of  good $\epsilon_0$ covers.

\begin{definition}\label{deff1}  Let $E\subset {Q_0}$ be given, let $\epsilon_0\in (0,1)$ and let $k$ be an integer. A good $\epsilon_0$ cover of $E$, of length $k$, is a collection
$\{\mathcal{O}_l\}_{l=1}^k$ of nested (relatively) open subsets of ${Q_0}$, together with collections
$\F_l=\{\Delta_i^l\}_i\subset Q_0$, $\Delta_i^l\in \mathbb D$,  such that
\begin{eqnarray}\label{cover1}
E\subset \mathcal{O}_k\subset\mathcal{O}_{k-1}\subset....\subset\mathcal{O}_1\subset Q_0,
\end{eqnarray}
\begin{eqnarray}\label{cover2}
\mathcal{O}_l=\bigcup_{\F_l}\Delta_i^l,
\end{eqnarray}
and
\begin{eqnarray}\label{cover3}\omega(\mathcal{O}_l\cap \Delta_i^{l-1})\leq \epsilon_0\omega(\Delta_i^{l-1}),\mbox{ for all }\Delta_i^{l-1}\in\F_{l-1}.
\end{eqnarray}
\end{definition}

Using the notion of  good $\epsilon_0$ covers we can reduce the proof of  Theorem \ref{Ainfty} to the proof of the following three lemmas.

\begin{lemma}\label{existcover} Let $E\subset {Q_0}$ be given, consider $\epsilon_0\in (0,1)$ and let $k$ be a positive integer. There exist  $\gamma=\gamma(m,\kappa,M_1)$, $0<\gamma\ll 1$, and $\Upsilon=\Upsilon(m,\kappa,M_1)$, $1\ll\Upsilon$, such that if we let $\delta_0=\gamma(\epsilon_0/\Upsilon)^k$, and if
$\omega(E)\leq\delta_0$,  then
$E$  has a good $\epsilon_0$ cover of length $k$.
\end{lemma}

\begin{lemma}\label{lemmacruc}  Let $\Upsilon\gg 1$ be given and consider $\delta_0\in (0,1)$. Assume that $E\subset {Q}_0$ with $\omega(E)\leq\delta_0$. If
$\delta_0=\delta_0(m,\kappa,M_1,\Upsilon)$ is chosen sufficiently small, then there exists a Borel set $S\subset\partial\Omega$, and a constant $c=c(m,\kappa,M_1)$, $1\leq c<\infty$,  such that if we let $u(Z,t):=\omega(Z,t,S)$, then
$$\Upsilon^2\sigma(E)\leq c \iiint_{T_{cQ_0}}|\nabla_Xu|^2\delta\, \d Z\d t.$$
Here
$\delta=\delta(Z,t)$ is as in \eqref{deltadist}, i.e. the distance from $(Z,t)\in \Omega$ to $\Sigma$, and $T_{cQ_0}$ is the Carleson set associated to $cQ_0$ as defined in \eqref{eq2.box}.
\end{lemma}

\begin{lemma}\label{lemmacruc+}  Let  $u(Z,t):=\omega(Z,t,S)$ and $c$ be as stated in Lemma \ref{lemmacruc}.
Then there exists $\tilde c=\tilde c(m,\kappa, M_1,M_2,\Gamma)$,  $1\leq \tilde c<\infty$,  such that
$$\iiint_{ T_{cQ_0}}|\nabla_Xu|^2\delta\, \d Z\d t\leq \tilde c\sigma(Q_0).$$
\end{lemma}

The proofs of Lemmas \ref{existcover}-\ref{lemmacruc+} are given in the forthcoming sections of the paper.  To prove Theorem \ref{Ainfty} using these auxiliary lemmas, we note that first using Lemma \ref{lemmacruc} and
Lemma \ref{lemmacruc+}  we can, for $\Upsilon\gg 1$ given, choose $\delta_0=\delta_0(m,M_1,\Upsilon)$, so that if
$E\subset {Q_0}$ with $\omega(E)\leq\delta_0$, then
\begin{eqnarray}
\Upsilon^2\sigma(E)\leq \hat c\sigma(Q_0),
\end{eqnarray}
for some $\hat c=\hat c(m,\kappa, M_1,M_2,\Gamma)$, $1\leq \hat c<\infty$. In particular,  we can conclude that
there exists, for every $\varepsilon>0$, a positive $\delta_0=\delta_0(m,\kappa,M_1,M_2,\Gamma,\varepsilon)$ such that
\begin{eqnarray}\omega(E)\leq\delta_0\leq c\delta_0\omega({Q_0})\implies\sigma(E)\leq \varepsilon\sigma({Q_0}),
\end{eqnarray}
where we have also applied Lemma \ref{bourg} stated below.  Theorem \ref{Ainfty} now follows from the doubling property of $\omega$, see Lemma \ref{T:doubling}, and the classical result in \cite{CF}.

 The rest of the paper is devoted to the proofs of  Lemmas \ref{existcover}-\ref{lemmacruc+} and we consider the proof of Lemma \ref{lemmacruc+} a rather  difficult part in the proof of Theorem \ref{Ainfty}. We here show how to reduce  Lemma \ref{lemmacruc+} to a core technical estimate. To prove Lemma \ref{lemmacruc+} we can without loss of generality assume that
$(Z_{Q_0},t_{Q_0})=(0,0)$ and we let $\rho_0:=l(Q_0)$. Throughout the rest of the paper we let $\P$ denote a parabolic approximation of the identity: $\P\in C_0^\infty(\B_1(0,0))$, $\B_1(0,0)\subset\mathbb R^{N-1}$,  $\P\geq 0$ is real-valued, and $\iint \P\, \d z \d t=1$. We will assume, as we may
by imposing a product structure on $\P$, that $\P$ is even in the sense that
\begin{eqnarray}\label{even}
\iint x_i\P(z,t)\, \d z\d t=\iint y_i\P(z,t)\, \d z\d t=\iint t\P(z,t)\, \d z\d t=0
\end{eqnarray}
for $i\in\{1,...,m-1\}$. We set $\P_\lambda(z,t)=\P_\lambda(x,y,t)=\lambda^{-{({\bf q}-4)}}\P(\lambda^{-1}x,\lambda^{-3}y,\lambda^{-2}t)$ whenever $\lambda>0$. Given $\P$ we let
$\P_\lambda$ define a convolution operator as introduced in \eqref{eq1vi}.  To prove Lemma \ref{lemmacruc+} we need to enable partial integration and we therefore use the mapping,
  \begin{eqnarray}\label{dom+ggaint}
  U \owns (w,w_m,y,y_m,t) \mapsto (w,w_m+\P_{\gamma w_m}\psi(w,y,t),y,y_m,t),
 \end{eqnarray} where
 %
 %
%
    \begin{eqnarray}\label{dom+gint}
 U&=&\{(W,Y,t)=(w, w_m,y,y_m, t)\in\mathbb R^{m-1}\times\mathbb R\times\mathbb R^{m-1}\times\mathbb R\times\mathbb R \mid w_m>0\}.
   \end{eqnarray}
We will need the following two lemmas proved in \cite{N1}. Lemma \ref{carlemma-} and Lemma \ref{carlemma} correspond to Lemma 2.1 and Lemma 2.2 in \cite{N1}, respectively.

    \begin{lemma}\label{carlemma-} Let $\psi$ be a function satisfying \eqref{Lip-++a1}  for some constant $0<M_1<\infty$, let
 $\gamma\in (0,1)$ and let  $\P_{\gamma w_m}\psi$ be defined as above for $w_m>0$. Let
 $\theta,\tilde\theta\geq 0$ be integers and let $(\phi_1,..,\phi_{m-1})$ and $(\tilde\phi_1,..,\tilde\phi_{m-1})$ denote multi-indices.
 Let $\ell:=(\theta+|\phi|+3|\tilde\phi|+2\tilde\theta)$. Then
  \begin{eqnarray}\label{con1}
\biggl |\frac {\partial^{\theta+|\phi|+|\tilde\phi|}}{\partial w_m^{\theta}\partial w^{\phi}\partial y^{\tilde\phi}} \biggl ((w\cdot\nabla_y-\partial_t)^{\tilde \theta}(\P_{\gamma w_m}\psi(w,y,t)) \biggr )\biggr |\leq c(m,l)\gamma^{1-(l-\theta)}w_m^{1-l}M_1,
 \end{eqnarray}
 whenever $(W,Y,t)\in U$.
 \end{lemma}

  \begin{lemma}\label{carlemma} Let $\psi$ be a function satisfying \eqref{Lip-++a1} and \eqref{Lip-++a2} for some constants $0<M_1,M_2<\infty$, let
  $\gamma\in (0,1)$ and let  $\P_{\gamma w_m}\psi$ be defined as above for $w_m>0$. Let
 $\theta,\tilde\theta\geq 0$ be integers and let $(\phi_1,..,\phi_{m-1})$ and $(\tilde\phi_1,..,\tilde\phi_{m-1})$ denote multi-indices. Let $\ell:=(\theta+|\phi|+3|\tilde\phi|+2\tilde\theta)$. Let
 \begin{eqnarray*}
\d\mu=\d\mu(W,Y,t):=\biggl |\frac {\partial^{\theta+|\phi|+|\tilde\phi|}}{\partial  w_m^{\theta}\partial  w^{\phi}\partial y^{\tilde\phi}} \biggl (( w\cdot\nabla_y-\partial_t)^{\tilde \theta}(\P_{\gamma  w_m}\psi( w,y,t)) \biggr )\biggr |^2 w_m^{2l-3} \d W\d Y\d t,
 \end{eqnarray*}
 be defined on $U$. Then
 \begin{eqnarray*}\label{con2}
 \mu(U\cap \mathcal{B}_r)\leq c(m,l,M_1,M_2)\gamma^{2-2(l-\theta)}r^{{\bf q}-1},
\end{eqnarray*}
for all balls $\mathcal{B}_r=\mathcal{B}_r(Z_0,t_0)\subset\mathbb R^{N+1}$ centered on $\partial U$, $r>0$.
 \end{lemma}

 Using Lemma \ref{carlemma-} we see that that there exists $\hat\gamma=\hat\gamma(m,M_1)\in (0,1)$ such that if $\gamma\in (0,\hat\gamma)$ then
    \begin{eqnarray}\label{1-1}\frac 1 2\leq 1+\frac {\partial}{\partial w_m}(\P_{\gamma  w_m}\psi)(w,y,t)\leq \frac 32,
    \end{eqnarray}
   whenever $(w,w_m,y,y_m,t)\in U$. This implies, in particular, that the map in \eqref{dom+ggaint}  is one-to-one.

   Defining $v$ as the pull-back of $u(Z,t):=\omega(Z,t,S)$ under the map in \eqref{dom+ggaint}, i.e.
    \begin{eqnarray}v(w,w_m,y,y_m,t):=u(w,w_m+\P_{\gamma w_m}\psi(w,y,t),y,y_m,t)),
 \end{eqnarray}
 we see that to prove Lemma \ref{lemmacruc+} it suffices to prove that
              \begin{eqnarray}\label{keyestalla}
 I_\epsilon :=\iiint_{\mathbb R^{N+1}_+}|\nabla_{W}v|^2\Psi_\epsilon^2w_m\, \d W\d Y\d t \leq c(m,\kappa, M_1, M_2,\Gamma)\rho_0^{{\bf q}-1},
    \end{eqnarray}
    for $\epsilon>0$ small, where $$\R^{N+1}_+=\R^{m-1}\times\lbrace w_m>0 \rbrace\times \R^{m}\times \R,$$
    and where $\Psi_\epsilon$ is a smooth cut-off function such that $\Psi_\epsilon\equiv 1$ on $$([-c,c]^m\times [-c,c]^m\times [-c,c])\cap \{w_m\geq 2\epsilon\},$$ and $\Psi_\epsilon\equiv 0$ on $$\bigl (([-2c,2c]^m\times [-2c,2c]^m\times [-2c,2c])\setminus ([-c,c]^m\times [-c,c]^m\times [-c,c])\bigr )\cap \{w_m<\epsilon\},$$ where $c=c(m,\kappa,M_1)\gg 1$.

    Furthermore, the  pull-back $v$ is a (weak) solution to
    \begin{eqnarray}\label{e-kolm-ndggha-int}
   \tilde{\L} v=\nabla_{W}\cdot (\tilde A\nabla_{W} v)+\tilde B\cdot\nabla_{W} v+ D\cdot\nabla_{Y,t} v=0
    \end{eqnarray}
    in $U$ where the $\tilde A=(\tilde a_{i,j})$ and $\tilde B=(\tilde b_i)$ depend on $\L$ and the pull-back map in \eqref{dom+ggaint}. Here and in the following  $\nabla_{Y,t}=(\nabla_Y,\partial_t)$ and $D$ is the vector valued function
    \begin{eqnarray}\label{e-kolm-ndggha-intD}
        D:=(w,w_m+\P_{\gamma w_m}\psi(w,y,t),-1).
    \end{eqnarray}

    Using that $\Omega\subset\mathbb R^{N+1}$ is an (unbounded) admissible ($y_m$-independent) Lipschitz domain  with constants $(M_1,M_2)$ in the sense of Definition \ref{car}, that $A$ satisfies \eqref{eq2} with constant $\kappa$,  \eqref{struct}, and Lemma \ref{carlemma-}, it follows that
    $\tilde A$ and $\tilde B$ are measurable and locally bounded satisfying
    \begin{eqnarray}\label{eq2++}
      \tilde\kappa^{-1}|\xi|^2\leq \sum_{i,j=1}^{m}\tilde a_{i,j}(W,Y,t)\xi_i\xi_j,\quad \ \ |\tilde A(W,Y,t)\xi\cdot\zeta|\leq \tilde\kappa|\xi||\zeta|,
    \end{eqnarray}
    for some $\tilde\kappa\in [1,\infty)$, and for all $\xi,\zeta\in \mathbb R^{m}$, $(W,Y,t)\in\mathbb R^{N+1}$, and
    \begin{eqnarray}\label{eq2++a}
      w_m|\nabla_W \tilde A(W,Y,t)|+w_m|\tilde B(W,Y,t)|\leq c.
    \end{eqnarray}
    Here $\tilde \kappa$ and $c$ depends on $m$, $\kappa$ and $M_1$ only. In addition it is important to note that $\tilde A$ is symmetric. Furthermore, using Lemma \ref{carlemma},  and that the measures $\mu_1$ and $\mu_2$ defined in \eqref{measure1} are Carleson measures on $\Omega$ with constant $\Gamma$, we see that if we introduce $\d\tilde \mu_i=\d\tilde\mu_i(W,Y,t)$, $i\in\{1,2,3\}$,
 \begin{equation}
 \begin{split}
\d\tilde\mu_1&:= |\nabla_W\tilde A|^2w_m\, \d W\d Y\d t,\\
\d\tilde\mu_2&:=|\tilde B|^2w_m\, \d W\d Y\d t,\\
\d\tilde\mu_3&:=|D\cdot\nabla_{Y,t}\tilde A|^2w_m^3\, \d W\d Y\d t,
 \end{split}
 \end{equation}
as  measures on $U$, then
 \begin{eqnarray}\label{con2+}
 \tilde\mu_i(U\cap \mathcal{B}_\rho(w_0,0,Y_0,t_0))\leq c(m,\kappa, M_1, M_2,\Gamma)\rho^{{\bf q}-1},
\end{eqnarray}
whenever $(w_0,0,Y_0,t_0)\in\partial U$, $\rho>0$, and $\mathcal{B}_\rho(w_0,0,Y_0,t_0)\subset\mathbb R^{N+1}$, and for $i\in\{1,2,3\}$. In particular, all measures in $\{\tilde\mu_i\}$ define Carleson measures on $U$. Furthermore, we emphasize that by our assumptions
      \begin{eqnarray}\label{e-kolm-ndggha-lla+jja}
 \mbox{$\tilde A$ and $\tilde B$ are independent of $y_m$}.
    \end{eqnarray}

To prove \eqref{keyestalla} it suffices to prove the following lemma.
    \begin{lemma}\label{Carleson}
Let $\sigma\in (0,1)$ be a given degree of freedom. Then there exists a finite constant $c=c(m,\kappa, M_1, M_2,\Gamma,\sigma)$, such that
		  \begin{eqnarray*}
		  I_{\epsilon}\leq \sigma I_{\epsilon}+c\rho_0^{{\bf q}-1}.
		  \end{eqnarray*}
\end{lemma}

Note that by construction $I_{\epsilon}$ is finite. The proof of Lemma \ref{Carleson} is given in the next section and from here on, and hence also in the proof, we will not indicate the dependence on $\epsilon$ and simply write $I$ for $I_\epsilon$ and we note -- and this is a consequence of the introduction of $\epsilon$ -- that no boundary terms will survive when we perform partial integration. In addition we will also here, with a slight abuse of notation, let $Z:=(W,Y)$ and $\d Z\d t:=\d W\d Y\d t$.

In the proof of Lemma \ref{Carleson} we will also use the quantities
 \begin{eqnarray}\label{e-kolm-ndggha-lla+gg}
 J&:=&\iiint_{\mathbb R^{N+1}_+}|D\cdot\nabla_{Y,t}v|^2\Psi^4w_m^3\, \d Z\d t,\notag\\
 K&:=&\sum_{i=1}^m\iiint_{\mathbb R^{N+1}_+}\ |\nabla_{W}(\partial_{w_i}v)|^2 \Psi^4w_m^3\, \d Z\d t,\notag\\
L&:=&\iiint_{\mathbb R^{N+1}_+}\ |\nabla_{Y}v|^2 \Psi^6w_m^5\, \d Z\d t,\\
 L_{i}&:=&
 \iiint_{\mathbb R^{N+1}_+}|\partial_{y_i} v|^2\Psi^6w_m^5\, \d Z\d t,\notag\\
 M&:=&\sum_{i=1}^m  \iiint_{\mathbb R^{N+1}_+} |\nabla_W(\partial_{y_i}v)|^2\Psi^8w_m^7\, \d Z\d t.\notag
      \end{eqnarray}
      In the rather technical proof to follow, the crucial estimate in the proof of Lemma \ref{Carleson} is stated in \eqref{sv} below and states that
      \begin{eqnarray*}
L_{m}\lesssim M^{1/2}J^{1/2}+I+J,
 \end{eqnarray*}
 where $\lesssim$ means that we can control the constants. This estimate uses, in a crucial way it seems, that $\psi$ and $A$, and hence $\tilde A$, do not depend on $y_m$. It seems that this additional  degree of freedom is crucial for us to be able to complete the argument.

\section{Proof of Lemma \ref{Carleson}}\label{sec5}
We will first prove that
            \begin{eqnarray}\label{auxest1}
            I\leq c\rho_0^{{\bf q}-1}+\sigma I+\tilde\sigma J
      \end{eqnarray}
      where $\sigma,\tilde\sigma\in (0,1)$ are degrees of freedom and $c$ is a positive constant  which, unless otherwise stated, only depends on $(m,\kappa, M_1, M_2,\Gamma)$ and $\sigma, \tilde\sigma$. In general, in the following $c$ will denote a generic such constant, not necessarily the same at each instance. We often write $c_1\lesssim c_2$ and this means that $c_1/c_2$ is bounded by a constant depending only on $(m,\kappa, M_1, M_2,\Gamma)$, $\sigma$ and $\tilde\sigma$.

To start the proof of \eqref{auxest1} we note, using  ellipticity, that
    \begin{eqnarray*}\label{est4}
   I\lesssim \sum_{i,j=1}^m I_{i,j},
    \end{eqnarray*}
    where
    \begin{align*}
             I_{i,j}:=2\iiint_{\mathbb R^{N+1}_+}\tilde a_{m,m}^{-1}\tilde a_{i,j}(\partial_{w_i}v)(\partial_{w_j}v)\Psi^2w_m\, \d Z\d t.
    \end{align*}
    Assume first that $i\neq m$. Then, integrating by parts in $I_{i,j}$ with respect to $w_i$ we see that
         \begin{eqnarray*}
             I_{i,j}&=&-2\iiint_{\mathbb R^{N+1}_+}\tilde a_{m,m}^{-1}v\partial_{w_i}(\tilde a_{i,j}\partial_{w_j}v)\Psi^2w_m\, \d Z\d t\notag\\
             &&-2\iiint_{\mathbb R^{N+1}_+}\partial_{w_i}\tilde a_{m,m}^{-1}\tilde a_{i,j}v(\partial_{w_j}v)\Psi^2w_m\, \d Z\d t\notag\\
             &&-4\iiint_{\mathbb R^{N+1}_+}\tilde a_{m,m}^{-1} {\tilde a_{i,j}}v(\partial_{w_j}v)w_m\Psi\partial_{w_i}\Psi\, \d Z\d t.
    \end{eqnarray*}
    Similarly we see that
             \begin{eqnarray*}
             I_{m,j}&=&-2\iiint_{\mathbb R^{N+1}_+}\tilde a_{m,m}^{-1}v\partial_{w_m}(\tilde a_{m,j}\partial_{w_j}v)\Psi^2w_m\, \d Z\d t\notag\\
             &&-2\iiint_{\mathbb R^{N+1}_+}\partial_{w_m}\tilde a_{m,m}^{-1}\tilde a_{m,j} v(\partial_{w_j}v)\Psi^2w_m\, \d Z\d t\notag\\
              &&-2\iiint_{\mathbb R^{N+1}_+}{\tilde a_{m,m}^{-1}}{\tilde a_{m,j}}v(\partial_{w_j}v)\Psi^2\, \d Z\d t\notag\\
               &&-4\iiint_{\mathbb R^{N+1}_+}{\tilde a_{m,m}^{-1}}{\tilde a_{m,j}}v(\partial_{w_j}v)w_m\Psi\partial_{w_m}\Psi\, \d Z\d t.
    \end{eqnarray*}
    Put together
        \begin{eqnarray*}
             I\leq I_1+I_2+I_3+I_4,
               \end{eqnarray*}
               where
    \begin{eqnarray*}
             I_1&:=&-2\sum_{i,j}\iiint_{\mathbb R^{N+1}_+}\tilde a_{m,m}^{-1}v\partial_{w_i}(\tilde a_{i,j}\partial_{w_j}v)\Psi^2w_m\, \d Z\d t,\notag\\
    I_2&:=&-2\sum_{i,j}\iiint_{\mathbb R^{N+1}_+}\partial_{w_i}\tilde a_{m,m}^{-1}\tilde a_{i,j}v(\partial_{w_j}v)\Psi^2w_m\, \d Z\d t,\notag\\
               I_3&:=&-4\sum_{i,j}\iiint_{\mathbb R^{N+1}_+}\tilde a_{m,m}^{-1} {\tilde a_{i,j}}v(\partial_{w_j}v)w_m\Psi\partial_{w_i}\Psi\, \d Z\d t,\notag\\
              I_4&:=&-2\sum_{j}\iiint_{\mathbb R^{N+1}_+}{\tilde a_{m,m}^{-1}}{\tilde a_{m,j}}v(\partial_{w_j}v)\Psi^2\, \d Z\d t.
               \end{eqnarray*}

               We first analyze $I_1$. Using the equation, i.e. \eqref{e-kolm-ndggha-int}, we obtain
                   \begin{eqnarray*}
             I_1&=&2\iiint_{\mathbb R^{N+1}_+}\tilde a_{m,m}^{-1}v (D\cdot\nabla_{Y,t}v)\Psi^2w_m\, \d Z\d t\notag\\
             &&+2\sum_{i}\iiint_{\mathbb R^{N+1}_+}\tilde a_{m,m}^{-1}vb_i\partial_{w_i}v\Psi^2w_m\, \d Z\d t\notag\\
             &=:&I_{11}+I_{12},
             \end{eqnarray*}
             and \begin{eqnarray*}
  I_{12}&\leq&c\biggl (\iiint_{\mathbb R^{N+1}_+}v^2|\tilde B|^2\Psi^2w_m\, \d Z\d t\biggr )^{1/2}
  \biggl (\iiint_{\mathbb R^{N+1}_+}|\nabla_W v|^2\Psi^2w_m\, \d Z\d t\biggr )^{1/2}\notag\\
  &\leq& c\rho_0^{{\bf q}-1}+\sigma I,
  \end{eqnarray*}
  by \eqref{con2+} applied to $\tilde\mu_2$. Furthermore, integrating by parts with respect to $w_m$ we see that
                        \begin{eqnarray*}
             I_{11}=I_{111}+I_{112}+I_{113}+I_{114},
             \end{eqnarray*}
             where
                      \begin{eqnarray*}
             I_{111}&=&-\iiint_{\mathbb R^{N+1}_+}\partial_{w_m}\tilde a_{m,m}^{-1}v(D\cdot\nabla_{Y,t}v) \Psi^2w_m^2\, \d Z\d t,\notag\\
             I_{112}&=&-\iiint_{\mathbb R^{N+1}_+}\tilde a_{m,m}^{-1}\partial_{w_m}v(D\cdot\nabla_{Y,t}v)  \Psi^2w_m^2\, \d Z\d t,\notag\\
             I_{113}&=&-\iiint_{\mathbb R^{N+1}_+}\tilde a_{m,m}^{-1}v\partial_{w_m}(D\cdot\nabla_{Y,t}v) \Psi^2w_m^2\, \d Z\d t,\notag\\
             I_{114}&=&-2\iiint_{\mathbb R^{N+1}_+}\tilde a_{m,m}^{-1}v(D\cdot\nabla_{Y,t}v) w_m^2\Psi \partial_{w_m}\Psi\, \d Z\d t.
             \end{eqnarray*}
             Focusing on $I_{111}$ we see that
             \begin{eqnarray*}
             I_{111}&=&\iiint_{\mathbb R^{N+1}_+}\biggl (\frac {\partial_{w_m} \tilde a_{m,m}}{\tilde a_{m,m}^2}\biggr )v(D\cdot\nabla_{Y,t}v)\Psi^2w_m^2\, \d Z\d t\notag\\
             &\leq&c\biggl (\iiint_{\mathbb R^{N+1}_+}|\partial_{w_m}\tilde a_{m,m}|^2v^2 w_m\, \d Z\d t\biggr)^{1/2}J^{1/2}\notag\\
             &\leq & c\rho_0^{{\bf q}-1}+\tilde\sigma J,
  \end{eqnarray*}
  by \eqref{con2+} applied to $\tilde\mu_1$. To continue we see that
       \begin{eqnarray*}
         I_{113}&=&-\iiint_{\mathbb R^{N+1}_+}\tilde a_{m,m}^{-1}v(D\cdot\nabla_{Y,t}\partial_{w_m}v) \Psi^2w_m^2\, \d Z\d t,\notag\\
             &&-\iiint_{\mathbb R^{N+1}_+}\tilde a_{m,m}^{-1}v(1+\partial_{w_m}\P_{\gamma w_m}\psi(w,y,t))(\partial_{y_m}v) \Psi^2w_m^2\, \d Z\d t\notag\\
             &=:&I_{1131}+I_{1132}.
  \end{eqnarray*}
  To estimate $I_{1132}$ we write
                 \begin{eqnarray*}
         I_{1132}&=&-\frac 12\iiint_{\mathbb R^{N+1}_+}\tilde a_{m,m}^{-1}(1+\partial_{w_m}\P_{\gamma w_m}\psi(w,y,t))(\partial_{y_m}v^2) \Psi^2w_m^2\, \d Z\d t\notag\\
             &=&\iiint_{\mathbb R^{N+1}_+}\tilde a_{m,m}^{-1}(1+\partial_{w_m}\P_{\gamma w_m}\psi(w,y,t)) v^2 w_m^2\Psi\partial_{y_m}\Psi\, \d Z\d t,
             \end{eqnarray*}
             where we have used that $\tilde a_{m,m}$ and $\psi$ are independent of $y_m$. In particular, $|I_{1132}|\leq c\rho_0^{{\bf q}-1}$. Focusing on
             $I_{1131}$,
                      \begin{eqnarray*}
         I_{1131}&=&\iiint_{\mathbb R^{N+1}_+}((D\cdot\nabla_{Y,t})\tilde a_{m,m}^{-1})v(\partial_{w_m}v) \Psi^2w_m^2\, \d Z\d t,\notag\\
             &&+\iiint_{\mathbb R^{N+1}_+}\tilde a_{m,m}^{-1}(D\cdot\nabla_{Y,t}v)(\partial_{w_m}v) \Psi^2w_m^2\, \d Z\d t\notag\\
             &&+2\iiint_{\mathbb R^{N+1}_+}\tilde a_{m,m}^{-1}v(\partial_{w_m}v) w_m^2\Psi(D\cdot\nabla_{Y,t})\Psi\, \d Z\d t\notag\\
         &=:&I_{11311}+I_{11312}+I_{11313}.
  \end{eqnarray*}
  Again using \eqref{con2+} applied to $\tilde\mu_3$ and elementary estimates we see that
                           \begin{eqnarray*}
       |I_{11311}|+|I_{11313}|\leq c\rho_0^{{\bf q}-1}+\sigma I.
  \end{eqnarray*}
   Furthermore, \begin{eqnarray*}
  I_{11312}=-I_{112}.
  \end{eqnarray*}
  In particular, we have proved that
          \begin{eqnarray*}
             I_1\leq c\rho_0^{{\bf q}-1}+\sigma I+\tilde\sigma J+|I_{114}|.
               \end{eqnarray*}

  To estimate $I_{114}$ we write
  \begin{eqnarray*}
                     I_{114}&=&-\iiint_{\mathbb R^{N+1}_+}\tilde a_{m,m}^{-1}(D\cdot\nabla_{Y,t}v^2)w_m^2\Psi \partial_{w_m}\Psi\, \d Z\d t\notag\\
             &=&\iiint_{\mathbb R^{N+1}_+}((D\cdot\nabla_{Y,t})\tilde a_{m,m}^{-1})v^2w_m^2\Psi \partial_{w_m}\Psi\, \d Z\d t\notag\\
             &&+\iiint_{\mathbb R^{N+1}_+}\tilde a_{m,m}^{-1}v^2w_m^2\Psi((D\cdot\nabla_{Y,t}) \partial_{w_m}\Psi)\, \d Z\d t\notag\\
             &&+\iiint_{\mathbb R^{N+1}_+}\tilde a_{m,m}^{-1}v^2w_m^2 ((D\cdot\nabla_{Y,t})\Psi)\partial_{w_m}\Psi\, \d Z\d t\notag\\
             &=:&I_{1141}+I_{1142}+I_{1143}.
             \end{eqnarray*}
             Using \eqref{con2+} applied to $\tilde\mu_3$, and by now familiar arguments, we see that $|I_{114}|\leq c\rho_0^{{\bf q}-1}$. Put together we can conclude that
               \begin{eqnarray*}
              I_1\leq c\rho_0^{{\bf q}-1}+\sigma I+\tilde\sigma J.
             \end{eqnarray*}
            It is straightforward to see that
      \begin{eqnarray*}
             |I_2|+|I_3|\leq c\rho_0^{{\bf q}-1}+ \sigma I. \end{eqnarray*}
             To estimate $I_4$ we write
      \begin{eqnarray*}
         I_4=-2\sum_{j}\iiint_{\mathbb R^{N+1}_+}{\tilde a_{m,m}^{-1}} {\tilde a_{m,j}}v(\partial_{w_j}v)\Psi^2\, \d Z\d t.\notag\\
               \end{eqnarray*}
             We first consider the term in the definition of $I_4$ which corresponds to $j=m$ and we note that
               \begin{eqnarray*}
         \biggl|\iiint_{\mathbb R^{N+1}_+}\partial_{w_m}(v^2)\Psi^2\, \d Z\d t\biggr |=2\biggl|\iiint_{\mathbb R^{N+1}_+}v^2\Psi\partial_{w_m}\Psi\, \d Z\d t\biggr|\leq  c\rho_0^{{\bf q}-1}.
               \end{eqnarray*}
          Next we consider the terms
                in the definition of $I_4$ which corresponds to $j\neq m$. By integration by parts we see that
                     \begin{equation*}
                     \begin{split}
         -2&\iiint_{\mathbb R^{N+1}_+}{\tilde a_{m,m}^{-1}} {\tilde a_{m,j}}v(\partial_{w_j}v)\partial_{w_m}(w_m)\Psi^2\, \d Z\d t\\
         &= 2\iiint_{\mathbb R^{N+1}_+}\partial_{w_m}({\tilde a_{m,m}^{-1}} {\tilde a_{m,j}})v\partial_{w_j}v\Psi^2w_m\, \d Z\d t\\
         &\quad +2\iiint_{\mathbb R^{N+1}_+}{\tilde a_{m,m}^{-1}} {\tilde a_{m,j}}\partial_{w_m} v\partial_{w_j}v\Psi^2w_m\, \d Z\d t\\
         &\quad +2\iiint_{\mathbb R^{N+1}_+}{\tilde a_{m,m}^{-1}} {\tilde a_{m,j}}v\partial_{w_mw_j}v\Psi^2w_m\, \d Z\d t\notag\\
         &\quad +4\iiint_{\mathbb R^{N+1}_+}{\tilde a_{m,m}^{-1}} {\tilde a_{m,j}}v\partial_{w_j}vw_m\Psi\partial_{w_m}\Psi\, \d Z\d t.
                     \end{split}
               \end{equation*}

               Let

                                    \begin{eqnarray*}
         I_{41}&:=&2\sum_{j\neq m}\iiint_{\mathbb R^{N+1}_+}{\tilde a_{m,m}^{-1}} {\tilde a_{m,j}}\partial_{w_m}v\partial_{w_j}v\Psi^2w_m\, \d Z\d t,\notag\\
         I_{42}&:=&2\sum_{j\neq m}\iiint_{\mathbb R^{N+1}_+}{\tilde a_{m,m}^{-1}} {\tilde a_{m,j}}v\partial_{w_mw_j}v\Psi^2w_m\, \d Z\d t.
               \end{eqnarray*}
               By the above deductions, and using by now familiar arguments, we can conclude that \begin{eqnarray*}
                                    |I_4-I_{41}-I_{42}|\leq c\rho_0^{{\bf q}-1}+\sigma I.
               \end{eqnarray*}

               To estimate $I_{42}$ we use that $j\neq m$. Integrating by parts
                            \begin{eqnarray*}
         I_{42}&=&-2\sum_{j\neq m}\iiint_{\mathbb R^{N+1}_+}\partial_{w_j}({\tilde a_{m,m}^{-1}} {\tilde a_{m,j}})v\partial_{w_m}v\Psi^2w_m\, \d Z\d t\notag\\
         &&-2\sum_{j\neq m}\iiint_{\mathbb R^{N+1}_+}{\tilde a_{m,m}^{-1}} {\tilde a_{m,j}}\partial_{w_j}v\partial_{w_m}v\Psi^2w_m\, \d Z\d t\notag\\
         &&-4\sum_{j\neq m}\iiint_{\mathbb R^{N+1}_+}{\tilde a_{m,m}^{-1}} {\tilde a_{m,j}}v\partial_{w_m}vw_m\Psi\partial_{w_j}\Psi\, \d Z\d t\notag\\
         &=:&I_{421}+I_{422}+I_{423}. \end{eqnarray*}
         Note that
                        \begin{eqnarray*}
I_{422}=-I_{41},
               \end{eqnarray*}
               and that
                                                       \begin{eqnarray*}
      |I_{421}|+|I_{423}|\leq c\rho_0^{{\bf q}-1}+\sigma I,
               \end{eqnarray*}
               by familiar arguments. Summarizing we can conclude that
               \begin{eqnarray*}
I&\lesssim&|I_1|+|I_2|+|I_3|+|I_4|\leq c\rho_0^{{\bf q}-1}+\sigma I+\tilde\sigma J,
             \end{eqnarray*}
             where $\sigma$, $\tilde\sigma$, are degrees of freedom.  This completes the proof of the estimate in \eqref{auxest1}.

             The next step is to estimate $J$ in a similar fashion.
\subsection{Estimating the term $J$}
              To estimate $J$ we write
      \begin{eqnarray*}\label{e-kolm-ndggha-lla+ggco}
    J=-\iiint_{\mathbb R^{N+1}_+}(D\cdot\nabla_{Y,t}v)
    \bigl (\nabla_{W}\cdot (\tilde A\nabla_{W} v)+\tilde B\cdot\nabla_{W}v\bigr)\Psi^4w_m^3\, \d Z\d t,
      \end{eqnarray*}
      and
      \begin{eqnarray}
J=J_{1}+J_{2}+J_{3}+J_{4},
\end{eqnarray}
where
\begin{eqnarray*}
J_{1}&:=&-\sum_{j}\iiint_{\mathbb R^{N+1}_+}\ (D\cdot\nabla_{Y,t}v)\partial_{w_m}(\tilde a_{m,j}\partial_{w_j}v)
\Psi^4w_m^3\, \d Z\d t,\notag\\
J_{2}&:=&-\sum_{i\neq m}\iiint_{\mathbb R^{N+1}_+}\ (D\cdot\nabla_{Y,t}v)\partial_{w_i}(\tilde a_{i,m}\partial_{w_m}v)
\Psi^4w_m^3\, \d Z\d t,\notag\\
J_{3}&:=&-\sum_{i\neq m}\sum_{j\neq m}\iiint_{\mathbb R^{N+1}_+}\ (D\cdot\nabla_{Y,t}v)
\partial_{w_i}(\tilde a_{i,j}\partial_{w_j}v) \Psi^4w_m^3\, \d Z\d t,\notag\\
J_{4}&:=&-\sum_{i}\iiint_{\mathbb R^{N+1}_+}\ (D\cdot\nabla_{Y,t}v){\tilde b_i}\partial_{w_i}v \Psi^4w_m^3\, \d Z\d t.
\end{eqnarray*}
Using \eqref{eq2++a} we immediately see that
\begin{eqnarray}\label{bound}
|J_{1}|+|J_{2}|+|J_{4}|\lesssim I^{1/2}J^{1/2}+\biggl (\iiint_{\mathbb R^{N+1}_+}\ |\nabla_W (\partial_{w_m}v)|^2 \Psi^4w_m^3\, \d Z\d t\biggr )^{1/2}J^{1/2}.
\end{eqnarray}

Focusing on $J_{3}$, and integrating by parts with respect to $w_i$, we see that
\begin{eqnarray}
J_{3}&=&\sum_{i\neq m}\sum_{j\neq m}\iiint_{\mathbb R^{N+1}_+}\ \partial_{w_i}(D\cdot\nabla_{Y,t}v)(\tilde a_{i,j}\partial_{w_j}v) \Psi^4w_m^3\, \d Z\d t\notag\\
&&+4\sum_{i\neq m}\sum_{j\neq m}\iiint_{\mathbb R^{N+1}_+}\ (D\cdot\nabla_{Y,t}v)(\tilde a_{i,j}\partial_{w_j}v) \partial_{w_i}(\Psi)w_m^3\Psi^3 \, \d Z\d t\\
&=:&J_{31}+J_{32},\notag
\end{eqnarray}
and that $|J_{32}|\leq cI^{1/2}J^{1/2}$. Furthermore,
\begin{eqnarray*}
J_{31}&=&\sum_{i\neq m}\sum_{j\neq m}\iiint_{\mathbb R^{N+1}_+}\ (\partial_{y_i}v)(\tilde a_{i,j}\partial_{w_j}v) \Psi^4w_m^3\, \d Z\d t\notag\\
&&+\sum_{i\neq m}\sum_{j\neq m}\iiint_{\mathbb R^{N+1}_+}\ (D\cdot\nabla_{Y,t}(\partial_{w_i}v))(\tilde a_{i,j}\partial_{w_j}v) \Psi^4w_m^3\, \d Z\d t\notag\\
&&+\sum_{i\neq m}\sum_{j\neq m}\iiint_{\mathbb R^{N+1}_+}\ (\partial_{w_i}\P_{\gamma w_m}\psi(w,y,t))(\partial_{y_m}v)(\tilde a_{i,j}\partial_{w_j}v) \Psi^4w_m^3\, \d Z\d t\notag\\
&=:&J_{311}+J_{312}+J_{313}.
\end{eqnarray*}
Then
\begin{eqnarray}
|J_{311}|+|J_{313}|\lesssim \biggl (\iiint_{\mathbb R^{N+1}_+}\ |\nabla_{Y}v|^2 \Psi^6w_m^5\, \d Z\d t\biggr )^{1/2} I^{1/2}.
\end{eqnarray}
To estimate $J_{312}$ we lift
the vector field  $D\cdot\nabla_{Y,t}$ through partial integration and
use the symmetry of the matrix $\{\tilde a_{i,j}\}$
to see that
\begin{eqnarray*}
2J_{312}&=&-\sum_{i\neq m}\sum_{j\neq m}\iiint_{\mathbb R^{N+1}_+}\ (\partial_{w_i}v)(D\cdot\nabla_{Y,t}(\tilde a_{i,j}))\partial_{w_j}v \Psi^4w_m^3\, \d Z\d t\notag\\
&&-4\sum_{i\neq m}\sum_{j\neq m}\iiint_{\mathbb R^{N+1}_+}\ (\partial_{w_i}v)(\tilde a_{i,j}\partial_{w_j}v)
(D\cdot\nabla_{Y,t}(\Psi))w_m^3\Psi^3\, \d Z\d t\notag\\
&=:&J_{3121}+J_{3122}.
\end{eqnarray*}
Then, by familiar arguments,
\begin{eqnarray}
|J_{3121}|+|J_{3122}|\leq cI.
\end{eqnarray}

Let $K$ and $L$ be as introduced in \eqref{e-kolm-ndggha-lla+gg}.  Then, putting all estimates together  we can conclude that
\begin{equation}
\begin{split}
J&\leq |J_{1}|+|J_{2}|+|J_{3}|+|J_{4}|\\
&\lesssim  I+I^{1/2}J^{1/2}+J^{1/2}K^{1/2}+I^{1/2}L^{1/2}.
\end{split}
\end{equation}
Hence
\begin{eqnarray}
J&\lesssim&  I+K+I^{1/2}L^{1/2}.
\end{eqnarray}
To proceed we have to estimate $K$ and $L$.

\subsection{Estimating the term $K$} To start the argument for $K$  we introduce $\tilde v=\partial_{w_i} v$ and we use \eqref{e-kolm-ndggha-int} to conclude that
 $\tilde v$ solves
\begin{equation}\label{e-kolm-nd+a}
\begin{split}
   &\nabla_{W}\cdot (\tilde A\nabla_{W} \tilde v)+\tilde B\cdot\nabla_{W} \tilde v+(D\cdot\nabla_{Y,t})\tilde v\\
   &=-\nabla_{W}\cdot (\partial_{w_i} \tilde A\nabla_{W} v)-\partial_{w_i} \tilde B\cdot\nabla_{W} v-\partial_{y_i}v-\partial_{w_i}\P_{\gamma w_m}\psi(w,y,t)\partial_{y_m}v
   \end{split}
    \end{equation}
    in $U$.
    Multiplying the equation in \eqref{e-kolm-nd+a} with $\tilde v\Psi^4w_m^3$, integrating and using Cauchy-Schwarz we see that
 \begin{eqnarray}\label{Lintbparts}
K\lesssim I+|K_1|+|K_2|+|K_3|+|K_4|,
\end{eqnarray}
where
\begin{eqnarray*}
K_1&:=&\iiint_{\mathbb R^{N+1}_+} \bigl ((D\cdot\nabla_{Y,t})\tilde v\bigr )\tilde v\Psi^4w_m^3\, \d Z\d t,\notag\\
K_2&:=& \iiint_{\mathbb R^{N+1}_+} \bigl (\nabla_{W}\cdot ((\partial_{w_i} \tilde A)\nabla_{W}v)\bigr )\tilde v\Psi^4w_m^3\, \d Z\d t,\notag\\
K_3&:=& \iiint_{\mathbb R^{N+1}_+} \bigl (\partial_{w_i}\tilde B\cdot \nabla_{W}v\bigr )\tilde v\Psi^4w_m^3\, \d Z\d t,\notag\\
K_4&:=&\iiint_{\mathbb R^{N+1}_+} (\partial_{y_i}v+\partial_{w_i}\P_{\gamma w_m}\psi(w,y,t)\partial_{y_m}v)\tilde v\Psi^4w_m^3\, \d Z\d t.
\end{eqnarray*}
Using \eqref{eq2++a} we immediately see that
 \begin{eqnarray}
|K_2|+|K_3|+|K_4|\lesssim I+I^{1/2}K^{1/2}+I^{1/2}L^{1/2}.
\end{eqnarray}
Furthermore,
 \begin{equation}\label{acom}
 \begin{split}
    2|K_1| &\leq 4\biggl |\iiint_{\mathbb R^{N+1}_+} \tilde v^2\bigl ((w,w_m+\P_{\gamma w_m}\psi(w,y,t))
    \cdot \nabla_Y-\partial_t\bigr)(\Psi)w_m^3\Psi^3\, \d Z\d t\biggr |\\
    &\lesssim I,
 \end{split}
\end{equation}
and we can conclude that
 \begin{eqnarray}
K\lesssim I+|K_1|+|K_2|+|K_3|+|K_4|\lesssim K\lesssim I+I^{1/2}K^{1/2}+I^{1/2}L^{1/2}.
\end{eqnarray}
Hence
 \begin{eqnarray}
K\lesssim I+I^{1/2}L^{1/2}.
\end{eqnarray}

\subsection{Estimating the term $L$ ($L_i$)}
Focusing on $L$ we write
 \begin{eqnarray}
 L=\sum_{i=1}^m L_{i}
\end{eqnarray}
where $L_i$ is defined in \eqref{e-kolm-ndggha-lla+gg}.
Note that
 \begin{eqnarray}\label{dyivrel}
 \partial_{y_i} v=-(D\cdot\nabla_{Y,t})(\partial_{w_i}v)+\partial_{w_i}(D\cdot\nabla_{Y,t})(v)-(\partial_{w_i}\P_{\gamma w_m}\psi(w,y,t))\partial_{y_m}v.
\end{eqnarray}
Hence,
 \begin{eqnarray*}
 L_{i}&=&-\iiint_{\mathbb R^{N+1}_+} (\partial_{y_i} v)(D\cdot\nabla_{Y,t})(\partial_{w_i}v)w_m^5\Psi^6\, \d Z\d t\notag\\
 &&+\iiint_{\mathbb R^{N+1}_+} (\partial_{y_i} v)\partial_{w_i}(D\cdot\nabla_{Y,t})(v)w_m^5\Psi^6\, \d Z\d t\notag\\
 &&-\iiint_{\mathbb R^{N+1}_+} (\partial_{y_i} v)((\partial_{w_i}\P_{\gamma w_m}\psi(w,y,t))\partial_{y_m}v)w_m^5\Psi^6\, \d Z\d t\notag\\
 &=:&L_{i,1}+L_{i,2}+L_{i,3}.
\end{eqnarray*}
Using partial integration 
we immediately see that
 \begin{eqnarray}
|L_{i,2}|\lesssim M^{1/2}J^{1/2}+L_{i}^{1/2}J^{1/2}
 \end{eqnarray}
where also $M$ was defined in \eqref{e-kolm-ndggha-lla+gg}. Furthermore,
\begin{eqnarray*}
L_{i,1}&=&\iiint_{\mathbb R^{N+1}_+} (D\cdot\nabla_{Y,t})(\partial_{y_i} v)(\partial_{w_i}v)w_m^5\Psi^6\, \d Z\d t\notag\\
&&+6\iiint_{\mathbb R^{N+1}_+} (\partial_{y_i} v)(\partial_{w_i}v)w_m^5\Psi^5(D\cdot\nabla_{Y,t})\Psi\, \d Z\d t\notag\\
&=&\iiint_{\mathbb R^{N+1}_+} \partial_{y_i} (D\cdot\nabla_{Y,t})(v)(\partial_{w_i}v)w_m^5\Psi^6\, \d Z\d t\notag\\
&&-\iiint_{\mathbb R^{N+1}_+} (\partial_{y_i} \P_{\gamma w_m}\psi(w,y,t))(\partial_{y_m} v)(\partial_{w_i}v)w_m^5\Psi^6\, \d Z\d t\notag\\
&&+6\iiint_{\mathbb R^{N+1}_+} (\partial_{y_i} v)(\partial_{w_i}v)w_m^5\Psi^5(D\cdot\nabla_{Y,t})\Psi\, \d Z\d t.
\end{eqnarray*}
Integrating by parts  we have
 \begin{eqnarray*}
L_{i,1} &=&-\iiint_{\mathbb R^{N+1}_+} (D\cdot\nabla_{Y,t})(v)(\partial_{w_iy_i}v)w_m^5\Psi^6\, \d Z\d t\notag\\
&&-6\iiint_{\mathbb R^{N+1}_+} (D\cdot\nabla_{Y,t})(v)(\partial_{w_i}v)w_m^5\Psi^5\partial_{y_i} \Psi\, \d Z\d t\notag\\
&&-\iiint_{\mathbb R^{N+1}_+} (\partial_{y_i} \P_{\gamma w_m}\psi(w,y,t))(\partial_{y_m} v)(\partial_{w_i}v)w_m^5\Psi^6\, \d Z\d t\notag\\
&&+6\iiint_{\mathbb R^{N+1}_+} (\partial_{y_i} v)(\partial_{w_i}v)w_m^5\Psi^5(D\cdot\nabla_{Y,t})\Psi\, \d Z\d t.
\end{eqnarray*}
Hence we can first conclude that
 \begin{eqnarray}
|L_{i,1}|\lesssim J^{1/2}M^{1/2}+I^{1/2}J^{1/2}+I^{1/2}(L_{i}^{1/2}+L_{m}^{1/2})
\end{eqnarray}
and then by collecting the estimates
 \begin{eqnarray}\label{hata}
L_{i}\lesssim I^{1/2}J^{1/2}+I^{1/2}(L_{i}^{1/2}+L_{m}^{1/2})+M^{1/2}J^{1/2}+L_{i}^{1/2}J^{1/2}+|L_{i,3}|.
\end{eqnarray}

We now first consider the case $i=m$. Using  \eqref{1-1} and the above we immediately see that
  \begin{eqnarray}\label{sv}
L_{m}\lesssim M^{1/2}J^{1/2}+I+J.
 \end{eqnarray}
 Consider now $i\neq m$. Then, using \eqref{hata} we have
  \begin{eqnarray}\label{hata+}
L_{i}\lesssim I^{1/2}J^{1/2}+I^{1/2}(L_{i}^{1/2}+L_{m}^{1/2})+M^{1/2}J^{1/2}+L_{i}^{1/2}J^{1/2}+L_{i}^{1/2}L_{m}^{1/2}.
\end{eqnarray}
Hence
  \begin{eqnarray}
L_{i}\lesssim I+I^{1/2}J^{1/2}+I^{1/2}L_{m}^{1/2}+M^{1/2}J^{1/2}+J+L_{m}
 \end{eqnarray}
 for $i\neq m$. In particular, using \eqref{sv}
   \begin{eqnarray}
L_{i}\lesssim I^{1/2}J^{1/2}+M^{1/2}J^{1/2}+I+J\mbox{ for all }i\in\{1,...,m\}.
 \end{eqnarray}

 Still the auxiliary term $M$ has to be estimated.

 \subsection{Estimating the term $M$}  To estimate $M$ we
introduce $\tilde v=\partial_{y_i} v$ and using the equation we see that
 $\tilde v$ solves
\begin{equation}\label{e-kolm-nd+auu}
\begin{split}
   &\nabla_{W}\cdot (\tilde A\nabla_{W} \tilde v)+\tilde B\cdot\nabla_{W} \tilde v+(D\cdot\nabla_{Y,t})\tilde v\\
   &=-\nabla_{W}\cdot (\partial_{y_i}\tilde A\nabla_{W} v)-\partial_{y_i}\tilde B\cdot\nabla_{W} v+\partial_{y_i}\P_{\gamma w_m}\psi(w,y,t)\partial_{y_m}v,
   \end{split}
\end{equation}
    in $U$. Multiplying this equation with $\tilde v w_m^7\Psi^8$  and arguing similarly as to the estimates in the case of the expression $K$ we derive
     \begin{eqnarray}\label{acomuu}
M\lesssim L+I^{1/2}L^{1/2}+M^{1/2}L^{1/2}+K^{1/2}L^{1/2}.
\end{eqnarray}
Hence,
     \begin{eqnarray}\label{acomuu+}
M\lesssim L+I^{1/2}L^{1/2}+K^{1/2}L^{1/2}.
\end{eqnarray}

\subsection{Completing the proof of Lemma \ref{Carleson}}
We are now ready to complete the proof of Lemma \ref{Carleson} by collecting all the terms and estimates developed above. To summarize we have proved that
\begin{equation}\label{acomuu++}
    \begin{split}
     I&\leq c\rho_0^{{\bf q}-1}+\sigma I+\tilde\sigma J,\\
     J&\lesssim   I+K+I^{1/2}L^{1/2},\\
     K&\lesssim I+I^{1/2}L^{1/2},\\
     L&\lesssim I^{1/2}J^{1/2}+M^{1/2}J^{1/2}+I+J,\\
    M&\lesssim L+I^{1/2}L^{1/2}+K^{1/2}L^{1/2}.
    \end{split}
\end{equation}
We again note that by construction of the test function $\Psi$ we can ensure that $I,\dots,M$ are finite. Using \eqref{acomuu++} we first see that
     \begin{eqnarray*}\label{acomuu++a}
     J+K&\lesssim&   I+\epsilon_1L,\notag\\
     L&\lesssim& I +J+\epsilon_2M,\notag\\
M&\lesssim& L+I+\epsilon_3K,
\end{eqnarray*}
where $\epsilon_1,\epsilon_2$ and $\epsilon_3$ are positive degrees of freedom.  Using the estimates for $L$ and $M$ we have
     \begin{eqnarray*}\label{acomuu++b}
     L&\lesssim& I +J+\epsilon_4K.
\end{eqnarray*}
Hence
     \begin{eqnarray*}\label{acomuu++c}
     J+K\leq   c(I+c\epsilon_1( I +J+\epsilon_4K)),
\end{eqnarray*}
and we can conclude that
     \begin{eqnarray*}\label{acomuu++d}
     J+K\lesssim  I.
\end{eqnarray*}
In particular,
\begin{eqnarray*}
     I\leq c\rho_0^{{\bf q}-1}+\sigma I+\tilde \sigma I
\end{eqnarray*}
and the proof is complete.
\qed

\section{Proof of auxiliary lemmas}\label{sec6}

In this section we prove Lemma \ref{existcover} and Lemma \ref{lemmacruc} and in  the proof  we need a number of estimates for non-negative solutions recently established in \cite{LN}.
Lemma \ref{lem4.7bol} and Lemma \ref{lem4.5-Kyoto1} below are Lemma 4.15 and Lemma 5.2 in \cite{LN}, respectively.

\begin{lemma}\label{lem4.7bol}   Let $\Omega\subset\mathbb R^{N+1}$ is an unbounded ($y_m$-independent) Lipschitz domain with constant $M_1$ in the sense of Definition \ref{car}. Then there exist
 $\Lambda=\Lambda(m,M_1)$, $1\leq \Lambda<\infty$,   $c=c(m,\kappa,M_1)$,  $1\leq c<\infty$, and $\gamma=\gamma(m,\kappa,M_1)$, $0<\gamma<\infty$, such that the following is true. Let $(Z_0,t_0)\in\partial\Omega$ and $r>0$.  Assume that  $u$ is a non-negative (weak) solution to $\L u=0$ in $\Omega\cap \mathcal{B}_{2r}(Z_0,t_0)$ and consider  $\rho$, $\tilde\rho$, $0<\tilde\rho\leq\rho<r/c$. Then
\begin{equation}\label{coneset-lem39}
\begin{split}
&u(A_{\tilde\rho,\Lambda}^+(Z_0,t_0))\leq c(\rho/\tilde\rho)^\gamma u(A_{\rho,\Lambda}^+(Z_0,t_0)),\\
&u(A_{\tilde\rho,\Lambda}^-(Z_0,t_0))\geq c^{-1} (\tilde\rho/\rho)^\gamma u(A_{\rho,\Lambda}^-(Z_0,t_0)).
\end{split}
\end{equation}
\end{lemma}

\begin{lemma}\label{lem4.5-Kyoto1} Let $\Omega\subset\mathbb R^{N+1}$ is an unbounded ($y_m$-independent) Lipschitz domain with constant $M_1$ in the sense of Definition \ref{car}.
Let $(Z_0,t_0)\in\partial\Omega$ and $r>0$. Let $\theta\in (0,1)$ be given. Then there exists  $c=c(m,\kappa,M_1,\theta)$,
$1 \leq c < \infty$, such that  following holds.  Assume that  $u$ is a non-negative (weak) solution to $\L u=0$ in $\Omega\cap \mathcal{B}_{2r}(Z_0,t_0)$,
vanishing continuously on $\partial\Omega\cap \mathcal{B}_{2r}(Z_0,t_0)$. Then
\begin{eqnarray}
\sup_{\Omega\cap \mathcal{B}_{r/c}(Z_0,t_0)}u\leq \theta\sup_{\Omega\cap \mathcal{B}_{2r}(Z_0,t_0)}u.
\end{eqnarray}
\end{lemma}

\begin{remark}\label{remnot}
Let $\Omega\subset\mathbb R^{N+1}$ be an unbounded ($y_m$-independent) Lipschitz domain with constant $M_1$ in the sense of Definition \ref{car}. The constants  $\Lambda=\Lambda(m,M_1)$, $1\leq \Lambda<\infty$,   $c=c(m,\kappa,M_1)$,  $1\leq c<\infty$, referred to in Lemma \ref{lem4.7bol} are fixed in \cite{LN}.  In the following we also let $\Lambda$ and $c$ be determined accordingly.
\end{remark}

\begin{lemma}\label{bourg} Let $\Omega\subset\mathbb R^{N+1}$ is an unbounded ($y_m$-independent) Lipschitz domain with constant $M_1$ in the sense of Definition \ref{car}. Let $\Lambda=\Lambda(m,M_1)$ be in accordance with Remark \ref{remnot}. Let $(Z_0,t_0)\in\partial\Omega$ and $r>0$. Then
 \begin{eqnarray*}
&&\omega( A_{r/c,\Lambda}^+(Z_0,t_0),\partial\Omega\cap \mathcal{B}_{r}(Z_0,t_0))\geq c^{-1}.
\end{eqnarray*}
\end{lemma}

\begin{proof}
This follows immediately from Lemma \ref{lem4.5-Kyoto1}.
\end{proof}

Lemmas \ref{T:doubling}-\ref{lemmacruc-} below are Theorem 3.6, Lemma 12.2 and Lemma 12.3 in \cite{LN}, respectively.

\begin{lemma}\label{T:doubling} Let $\Omega\subset\mathbb R^{N+1}$ is an unbounded ($y_m$-independent) Lipschitz domain with constant $M_1$ in the sense of Definition \ref{car}. Let $\Lambda=\Lambda(m,M_1)$ be in accordance with Remark \ref{remnot}. Let $(Z_0,t_0)\in\partial\Omega$ and $r>0$. Then there exists  $c=c(m,\kappa,M_1)$,  $1\leq c<\infty$,  such that
 \begin{eqnarray*}
&&\omega( A_{r,\Lambda}^+(Z_0,t_0),\partial\Omega\cap \mathcal{B}_{2\tilde r}(\tilde Z_0,\tilde t_0))\leq c
\omega( A_{r,\Lambda}^+(Z_0,t_0),\partial\Omega\cap \mathcal{B}_{\tilde r}(\tilde Z_0,\tilde t_0))
\end{eqnarray*}
whenever $(\tilde Z_0,\tilde t_0)\in\partial\Omega$, $\mathcal{B}_{\tilde r}(\tilde Z_0,\tilde t_0)\subset \mathcal{B}_{r/c}(Z_0,t_0)$.
\end{lemma}

\begin{lemma}\label{lem4.5-Kyoto1ha}Let $\Omega\subset\mathbb R^{N+1}$ is an unbounded ($y_m$-independent) Lipschitz domain with constant $M_1$ in the sense of Definition \ref{car}. Let $\Lambda=\Lambda(m,M_1)$ be in accordance with Remark \ref{remnot}. Let $(Z_0,t_0)\in\partial\Omega$ and $r>0$. Let $(\tilde Z_0,\tilde t_0)\in\partial\Omega$ and $\tilde r>0$ be such that $\mathcal{B}_{\tilde r}(\tilde Z_0,\tilde t_0)\subset \mathcal{B}_{r}(Z_0,t_0)$.  Then there exists  $c=c(m,\kappa,M_1)$,  $1\leq c<\infty$, such that
 \begin{eqnarray}\label{ad}
K(A_{c\tilde r,\Lambda}^+(\tilde Z_0,\tilde t_0),\bar Z,\bar t):=\lim_{\bar r\to 0}\frac{\omega(A_{c\tilde r,\Lambda}^+(\tilde Z_0,\tilde t_0),\partial\Omega\cap\mathcal{B}_{\bar r}(\bar Z,\bar t))}{\omega(A_{cr,\Lambda}^+(Z_0,t_0),\partial\Omega\cap\mathcal{B}_{\bar r}(\bar Z,\bar t))}
\end{eqnarray}
exists for $\omega(A_{cr,\Lambda}^+(Z_0,t_0),\cdot)$-a.e. $(\bar Z,\bar t)\in \partial\Omega\cap\mathcal{B}_{\tilde r}(\tilde Z_0,\tilde t_0)$, and
\begin{eqnarray}\label{ad1}
c^{-1}\leq {\omega(A_{cr,\Lambda}^+(Z_0,t_0), \partial\Omega\cap\mathcal{B}_{\tilde r}(\tilde Z_0,\tilde t_0))}K(A_{c\tilde r,\Lambda}^+(\tilde Z_0,\tilde t_0),\bar Z,\bar t)\leq c
\end{eqnarray}
whenever $(\bar Z,\bar t)\in \partial\Omega\cap\mathcal{B}_{\tilde r}(\tilde Z_0,\tilde t_0)$.
\end{lemma}

\begin{lemma}\label{lemmacruc-} Let $\Omega\subset\mathbb R^{N+1}$ is an unbounded ($y_m$-independent) Lipschitz domain with constant $M_1$ in the sense of Definition \ref{car}. Let $\Lambda=\Lambda(m,M_1)$ be in accordance with Remark \ref{remnot}. Let $(Z_0,t_0)\in\partial\Omega$ and $r>0$. Let $(\tilde Z_0,\tilde t_0)\in\partial\Omega$ and $\tilde r>0$ be such that $\mathcal{B}_{\tilde r}(\tilde Z_0,\tilde t_0)\subset \mathcal{B}_{r}(Z_0,t_0)$.  Then there exist  $c=c(m,\kappa,M_1)$,  $1\leq c<\infty$, and $\tilde c=\tilde c(m,\kappa,M_1)$,  $1\leq \tilde c<\infty$, such that
 \begin{eqnarray*}
 \quad \tilde c^{-1}\omega(A_{c\tilde r,\Lambda}^+(\tilde Z_0,\tilde t_0),E)\leq \frac {\omega(A_{cr,\Lambda}^+(Z_0,t_0),E)}{\omega(A_{cr,\Lambda}^+(Z_0,t_0),\partial\Omega\cap \mathcal{B}_{\tilde r}(\tilde Z_0,\tilde t_0))}\leq \tilde c \omega(A_{c\tilde r,\Lambda}^+(\tilde Z_0,\tilde t_0),E),
\end{eqnarray*}
whenever $E\subset \mathcal{B}_{\tilde r}(\tilde Z_0,\tilde t_0)$.
\end{lemma}

\subsection{Proof of Lemma \ref{existcover}} Let in the following  $Q_0\in \mathbb{D}$, $\epsilon_0\in (0,1)$, and let $\omega(\cdot)$ be as in the statement of Theorem \ref{Ainfty}. Let $k\in \mathbb Z_+$ be given. Let $\gamma$, $0<\gamma\ll 1$, and $\Upsilon$, $1\ll\Upsilon$,  be degrees of freedom to be chosen depending only on $m$, $\kappa$ and $M_1$. Let $\delta_0=\gamma(\epsilon_0/\Upsilon)^k$. Suppose that $\omega(E)\leq\delta_0$. Using that $\omega$ is a regular Borel measure, we see that there exists a (relatively) open subset of ${Q}_0$, containing $E$, which we denote by $\mathcal{O}_{k+1}$, satisfying
$\omega(\mathcal{O}_{k+1})\leq 2\omega(E)$. Using Lemma \ref{bourg} and the Harnack inequality, see Lemma \ref{lem4.7bol}, we see that there exists $c=c(m,\kappa,M_1)$,  $1\leq c<\infty$, such that
\begin{eqnarray}\label{yy1}
\omega(\mathcal{O}_{k+1})\leq 2\delta_0\leq c\delta_0\omega({Q}_0)\leq \frac 1 2 \biggl (\frac {\epsilon_0}{\Upsilon}\biggr )^k\omega({Q}_0)
\end{eqnarray}
if we let $\gamma:=1/(2c)$.  Let $f\in L^1_{\mbox{loc}}(\Sigma,\d\omega)$, and let
$$M_{\omega}(f)(Z,t):=\sup_{\{\mathcal{B}_r(\tilde Z,\tilde t):\ (\tilde Z,\tilde t)\in\Sigma,\  (Z,t)\in\mathcal{B}_r(\tilde Z,\tilde t)\}}\frac 1{\omega(\mathcal{B}_r(\tilde Z,\tilde t))}\iint_{\mathcal{B}_r(\tilde Z,\tilde t)} |f|\, \d\omega,$$
denote the Hardy-Littlewood maximal function of $f$, with respect to $\omega$, and
where the supremum is taken over all  balls $\mathcal{B}_r(\tilde Z,\tilde t)$, $(\tilde Z,\tilde t)\in\Sigma$, containing $(Z,t)$. Set
\begin{equation}\label{lem:coverexists_Okdef}
    \mathcal{O}_k:=\{(Z,t)\in {Q}_0\mid M_{\omega}(\chi_{\mathcal{O}_{k+1}})\geq {\epsilon_0}/{\bar c}\},
\end{equation}
where $\chi_{\mathcal{O}_{k+1}}$ denotes the indicator function for the set $\mathcal{O}_{k+1}$, and where we let $\bar c=\bar c(m,\kappa,M_1)$,  $1\leq \bar c<\infty$, denote the constant appearing in Lemma \ref{T:doubling}. Then, by construction,
$\mathcal{O}_{k+1}\subset\mathcal{O}_k$, $\mathcal{O}_k$ is  relatively open in
${Q}_0$ and $\mathcal{O}_k$ is properly contained in ${Q}_0$. As $\omega$ is doubling, see Lemma \ref{T:doubling},  we have that $(2Q_0,d,\omega)$ is a space of homogeneous type and weak $L^1$ estimates for the Hardy-Littlewood maximal function apply. Hence
\begin{eqnarray}\label{yy2}\omega(\mathcal{O}_k)\leq \tilde c\frac {\bar c}{\epsilon_0}\omega(\mathcal{O}_{k+1})\leq \frac 1 2 \biggl (\frac {\epsilon_0}{\bar c}\biggr )^{k-1}\omega({Q}_0),
\end{eqnarray}
if we let $\Upsilon:=\tilde c\bar c$ and where $\tilde c=\tilde c(m,\kappa,M_1)$,  $1\leq \tilde c<\infty$. By definition and by the construction, see \eqref{cubes} $(i)$-$(iii)$,  ${Q}_0$ can be dyadically subdivided, and we can select a collection
$\F_k=\{\Delta_i^k\}_i\subset {{Q}_0}$, comprised
of the cubes that are maximal with respect to containment in $\mathcal{O}_k$, and thus $\mathcal{O}_k := \cup_i \Delta^k_i$. The cubes in $\F_k$ are maximal in the sense that
\begin{equation}\label{Fkmaximal}
 \Delta_i^k \in \F_k \:\iff\:\Delta_i^k\subset\mathcal{O}_k \:\text{and}\: Q \subset \Delta_i^k, \:\forall Q\in\mathbb{D}_{Q_0}\:\text{such that}\:Q\subset\mathcal{O}_k.
\end{equation}
Using \eqref{Fkmaximal}, \eqref{lem:coverexists_Okdef}, and Lemma \ref{T:doubling}, we see that
\begin{equation}\label{yy3}
\begin{split}
        \omega(\mathcal{O}_{k+1}\cap\Delta_i^k) &\leq \omega(\mathcal{O}_{k+1}\cap 2\Delta_i^k)\\
        &\leq \omega(2\Delta_i^k)\frac{1}{\omega(2\Delta_i^k)}\iint_{2\Delta_i^k}\chi_{\mathcal{O}_{k+1}}\d\omega\\&\leq \epsilon_0\omega(\Delta_i^k),
\end{split}
\end{equation}
for all $\Delta_i^k\in\F_k$. We now iterate this argument, to construct $\mathcal{O}_{j-1}$ from $\mathcal{O}_j$, for $2\leq j\leq k$, just as we constructed $\mathcal{O}_k$ from $\mathcal{O}_{k+1}$. It is then a routine matter to verify that the sets $\mathcal{O}_1$,...., $\mathcal{O}_k$, form a good $\epsilon_0$ cover of $E$. We omit further details.
\qed

\subsection{Additional notation}

\begin{remark}\label{gc} In the following we let $\Pi(Z,t)$ denote the projection of $(Z,t)\in\mathbb R^{N+1}$ along $x_m$ onto $\partial\Omega$. Furthermore, from now on we fix two small dyadic numbers $\eta_1=2^{-k_1}$ and $\eta_2=2^{-k_2}$ where $1\leq k_1\ll k_2$ are to be chosen depending at most on $m$, $\kappa$ and $M_1$. Given $Q\in \mathbb D$, we let $A_{\eta_1 Q}^+:=A_{c\eta_1 l(Q),\Lambda}^+(Z_Q,t_Q)$, we consider the point
$A_{c\eta_1^2 l(Q),\Lambda}^-(Z_Q,t_Q)$ and we let $\tilde Q\in \mathbb D$ be such that $l(\tilde Q)=\eta_1^2 l(Q)$ and such that $\tilde Q$ contains the point $\Pi(A_{c\eta_1^2 l(Q),\Lambda}^-(Z_Q,t_Q))$. We can and will choose $\eta_1$ so small that
$$\Pi(A_{\eta_1 Q}^+)\subset \frac 1 4 Q$$
and such that
$$\tilde Q\subset Q.$$
\end{remark}

\begin{remark}\label{gc+}  Given $Q\in \mathbb D$ and $A_{\eta_1 Q}^+=A_{c\eta_1l(Q),\Lambda}^+(Z_Q,t_Q)$ as in Remark \ref{gc}, consider the point $\Pi(A_{\eta_1 Q}^+)\in\partial\Omega$. We let
$\hat Q\in \mathbb D$ be such that $l(\hat Q)=\eta_2 l(Q)$ and such that $\hat Q$ contains the point $\Pi(A_{\eta_1 Q}^+)$. We let $d_Q:=|A_{\eta_1 Q}^+-\Pi(A_{\eta_1 Q}^+)|$. Furthermore, we let
$\bar Q\in \mathbb D$ be such that $l(\bar Q)=\eta_2^2 l(\hat Q)$ and $\bar Q$ contains the point $(Z_{\hat Q},t_{\hat Q})$. Note that by construction, if we choose $\eta_1$  (and hence $\eta_2$) small enough,
$$\bar Q\subset\hat Q\subset Q.$$
Given
$\bar Q$ we let
\begin{equation}\label{yy4-}
\begin{split}
&S_{\bar Q}^+:=\{(x,\psi(x,y,t)+d_Q,y,y_m,t)\mid\ (x,\psi(x,y,t),y,y_m,t)\in \bar Q\},\\
&S_{\bar Q}^-:=\{(x,\psi(x,y,t)+l(\bar Q),y,y_m,t)\mid\ (x,\psi(x,y,t),y,y_m,t)\in \bar Q\}.
\end{split}
\end{equation}
Then $S_{\bar Q}^-$ and $S_{\bar Q}^+$ are two pieces of surfaces above $\partial\Omega$  in the direction of $x_m$, $S_{\bar Q}^+$ is above $S_{\bar Q}^-$, and each point on $S_{\bar Q}^-$ can be connected to a point on $S_{\bar Q}^+$ by a straight line in the direction of $x_m$.
\end{remark}

\begin{remark} Note that Remark \ref{gc} and Remark \ref{gc+} are generic constructions for dyadic cubes. Consider now the special case  $\Delta := \Delta_i^l\in\F_l$, i.e. $\Delta$ is a cube arising in some good $\epsilon_0$ cover. We then set $\tilde \Delta_i^l:=\tilde \Delta$, where $\tilde \Delta$ is defined as in Remark \ref{gc}, and we define
\begin{eqnarray}\label{yy4}\tilde \O_l:=\bigcup_{\Delta_i^l\in\F_l}\tilde \Delta_i^l.
\end{eqnarray}
Furthermore, let $E\subset {Q}_0$ and consider the set up of Lemma \ref{existcover}. We note that for every $(Z_0,t_0)\in E$ we have $(Z_0,t_0)\in \mathcal{O}_l$, for all $l=1,2,...,k$,
and that therefore there exists, for each $l$, a cube $\Delta_i^l=\Delta_i^l(Z_0,t_0)\in\mathcal{F}_l$ containing $(Z_0,t_0)$.
\end{remark}


\subsection{Proof of Lemma \ref{lemmacruc}} Let in the following  $Q_0\in \mathbb{D}$ and let $\omega(\cdot)$ be as in the statement of Theorem \ref{Ainfty}. To prove Lemma \ref{lemmacruc}, let $\epsilon_0>0$ be a  degree of freedom to be specified below and depending only on $m,\kappa,M_1$, let
 $\delta_0=\gamma(\epsilon_0/\Upsilon)^k$ be as specified in Lemma \ref{existcover} where $k$ is to be chosen depending only on $m,\kappa,M_1$ and $\Upsilon$. Consider $E\subset {Q_0}$ with $\omega(E)\leq\delta_0$.  Using Lemma \ref{existcover} we see that  $E$ has a good $\epsilon_0$ cover of length $k$, $\{\mathcal{O}_l\}_{l=1}^k$ with corresponding collections
$\F_l=\{\Delta_i^l\}_i\subset Q_0$. Let $\{\tilde{\mathcal{O}}_l\}_{l=1}^k$ be defined as in \eqref{yy4}. Using this good $\epsilon_0$ cover of $E$ we let
$$F(Z,t):=\sum_{j=2}^k \chi_{\tilde\O_{j-1}\setminus\O_j}(Z,t),$$
where $\chi_{\tilde\O_{j-1}\setminus\O_j}$  is the indicator function for the set $\tilde\O_{j-1}\setminus\O_j$. Then $F$ equals the indicator function of some Borel set $S\subset\Sigma$ and we let $u(Z,t):=\omega(Z,t, S)$. Consider \mbox{$(Z_0,t_0)\in E$} and  an index \mbox{$l\in \{1,...,k\}$}. Let in the following $$\mbox{$\Delta_i^l\in \F_l$ be a cube in the collection $\F_l$ which contains $(Z_0,t_0)$.}$$ Given $k_0\in\mathbb Z_+$ and $A_{\eta\Delta_i^l}^+=A_{c\eta l(\Delta_i^l),\Lambda}^+(Z_{\Delta_i^l},t_{\Delta_i^l})$ we let
$$\mbox{$\hat\Delta_i^l$, $\bar\Delta_i^l$, $S_{\bar\Delta_i^l}^-$ and $S_{\bar\Delta_i^l}^+$}$$
be as defined as in Remark \ref{gc+} relative to $\Delta_i^l$ and using $\eta_j:=2^{-k_j}$.
Hence, based on $(Z_0,t_0)\in E$ and an index $l\in \{1,...,k\}$, we have specified $\Delta_i^l$, $\hat \Delta_i^l$, $\bar \Delta_i^l$ and the surfaces
$S_{\bar\Delta_i^l}^-$ and $S_{\bar\Delta_i^l}^+$, and, by construction,
$$ \bar \Delta_i^l\subset\hat \Delta_i^l\subset  \Delta_i^l,\ \tilde  \Delta_i^l\subset  \Delta_i^l.$$
We first intend to prove that there exists $\beta>0$, depending only on $m,\kappa,M_1$,
such that if $\epsilon_0$ and $\eta_j=2^{-k_j}$ are chosen sufficiently small, then
\begin{eqnarray}\label{yy6}u(P+d_{\Delta_i^l}e_m)-u( P+l(\bar\Delta_i^l)e_m)\geq \beta,\ \forall P\in\{(x,\psi(x,y,t),y,y_m,t)\in \bar\Delta_i^l\},
\end{eqnarray}
where $e_m$ denotes the unit vector in $\mathbb R^{N+1}$ which points in the direction of $x_m$. Given $P\in\{(x,\psi(x,y,t),y,y_m,t)\in \bar\Delta_i^l\}$ we let
$$P^+:=P+d_{\Delta_i^l}e_m,\ P^-:=P+l(\bar\Delta_i^l)e_m,$$
and we want to estimate $u(P^+)-u(P^-)$. To start the proof we note that,  by construction,  $\tilde \Delta_i^l\subset\Delta_i^l$ and by using Lemma \ref{bourg} and Lemma \ref{lem4.7bol} we see that
there exists $ c_{\eta_1}=c_{\eta_1}(m,\kappa,M_1,\eta_1)$, $1\leq  c_{\eta_1}<\infty$, for $\eta_1$ small enough,
such that
\begin{eqnarray}\label{boundbelow}\omega(P^+,\tilde \Delta_i^l)\geq  c_{\eta_1}^{-1}.
\end{eqnarray}
To estimate $u( P^+)$  we first note that
\begin{equation}
\begin{split}
u(P^+)&\geq\iint_{\tilde \Delta_i^l}\chi_{\tilde{\mathcal{O}}_{l}\setminus\mathcal{O}_{l+1}}\, \d\omega(P^+,\bar Z,\bar  t)\\
 &=\omega(P^+,\tilde \Delta_i^l)-\omega( P^+,\tilde\Delta_i^l\cap \mathcal{O}_{l+1}),
 \end{split}
\end{equation}
as all terms in the definition of $F$ are non-negative. Consider $(\bar Z,\bar  t)\in \Delta_i^l$. Then, using Lemma \ref{lem4.5-Kyoto1ha} we have that
\begin{eqnarray}
K( P^+,\bar Z,\bar  t):=\lim_{\rho\to 0}\frac{\omega( P^+,\partial\Omega\cap\mathcal{B}_\rho(\bar Z,\bar  t))}{\omega(\partial\Omega\cap\mathcal{B}_\rho(\bar Z,\bar  t))},
\end{eqnarray}
exists for $\omega$-a.e. $(\bar Z,\bar  t)\in \Delta_i^l$, and
\begin{eqnarray}
K(P^+,\bar Z,\bar  t)\leq \frac c{\omega(\Delta_i^l)}\mbox{ whenever }(\bar Z,\bar  t)\in \Delta_i^l.
\end{eqnarray}
In the last conclusion we have also used Lemma \ref{T:doubling}. Using this estimate, and the fact that by construction $\tilde \Delta_i^l\subset\Delta_i^l$, we see that
\begin{eqnarray}
\omega( P^+,\tilde\Delta_i^l\cap \mathcal{O}_{l+1})\leq  \frac {c_{\eta_1}}{\omega(\Delta_i^l)}\omega(\tilde\Delta_i^l\cap \mathcal{O}_{l+1})\leq C_{\eta_1}\epsilon_0,
\end{eqnarray}
by the construction. In particular, using \eqref{boundbelow} we deduce
\begin{eqnarray}\label{yy10}
u(P^+)\geq  c_{\eta_1}^{-1}-C_{\eta_1}\epsilon_0.
\end{eqnarray}
To estimate $u( P^-)$ we write
\begin{eqnarray*}
u( P^-)&=&\iint_{Q_0\setminus \hat\Delta_i^l}F(\bar Z,\bar  t)\, \d\omega( P^-,\bar Z,\bar  t)+\iint_{\hat\Delta_i^l}F(\bar Z,\bar  t)\, \d\omega( P^-,\bar Z,\bar  t)\notag\\
&=:&I+{II}.
\end{eqnarray*}
Using Lemma \ref{lem4.5-Kyoto1} and the definition of $ P^-$ we see that
\begin{eqnarray}
|I|\leq \omega(P^-,Q_0\setminus \hat\Delta_i^l)\leq c\eta_2^\sigma,
\end{eqnarray}
for some $c=c(m,\kappa,M_1)$, $\sigma=\sigma(m,\kappa,M_1)\in (0,1)$.
We split ${II}$ as
\begin{eqnarray}
{II}= {II}_1+{II}_2+{II}_3,
\end{eqnarray}
where
\begin{eqnarray}
{II}_1&:=& \sum_{j=2}^l\iint_{\hat\Delta_i^l}1_{\tilde{\mathcal{O}}_{j-1}\setminus\mathcal{O}_j}\, \d\omega( P^-,\bar Z,\bar  t),\notag\\
{II}_2&:=&\sum_{j=l+2}^k\iint_{\hat\Delta_i^l}1_{\tilde{\mathcal{O}}_{j-1}\setminus\mathcal{O}_j}\, \d\omega( P^-,\bar Z,\bar  t),\\
{II}_3&:=&\iint_{\hat\Delta_i^l}1_{\tilde{\mathcal{O}}_{l}\setminus\mathcal{O}_{l+1}}\, \d\omega( P^-,\bar Z,\bar  t).\notag
\end{eqnarray}
Note that if $j\leq l$, then  $\hat\Delta_i^l\subset\Delta_i^l\subset\mathcal{O}_l\subset \mathcal{O}_j$ and $(\tilde{\mathcal{O}}_{j-1}\setminus\mathcal{O}_j)\cap \Delta_i^l=\emptyset$. Hence $II_1 = 0$. Furthermore,
\begin{equation}
\begin{split}
|II_2|&\leq\sum_{j=l+2}^k\omega( P^-,(\tilde{\mathcal{O}}_{j-1}\setminus\mathcal{O}_j)\cap \Delta_i^l)\\
&\leq c_{\eta_2}\sum_{j=l+2}^k\omega( A_{\hat\Delta_i^l}^+(\Pi(A_{\eta\Delta_i^l}^+)),(\tilde{\mathcal{O}}_{j-1}\setminus\mathcal{O}_j)\cap \Delta_i^l),
\end{split}
\end{equation}
by the Harnack inequality, see Lemma \ref{lem4.7bol}. Consider $(\bar Z,\bar  t)\in (\tilde{\mathcal{O}}_{j-1}\setminus\mathcal{O}_j)\cap \Delta_i^l$. Then, again using Lemma \ref{lem4.5-Kyoto1ha} we have that
\begin{eqnarray}
K( A_{\hat\Delta_i^l}^+(\Pi(A_{\eta\Delta_i^l}^+)),\bar Z,\bar  t):=\lim_{\rho\to 0}\frac{\omega( A_{\hat\Delta_i^l}^+(\Pi(A_{\eta\Delta_i^l}^+)),\partial\Omega\cap\mathcal{B}_\rho(\bar Z,\bar  t))}{\omega(\partial\Omega\cap\mathcal{B}_\rho(\bar Z,\bar  t))},
\end{eqnarray}
exists for $\omega$-a.e. $(\bar Z,\bar  t)\in \Delta_i^l$, and
\begin{eqnarray}
K(A_{\hat\Delta_i^l}^+(\Pi(A_{\eta\Delta_i^l}^+)),\bar Z,\bar  t)\leq \frac c{\omega(\Delta_i^l)}\mbox{ whenever }(\bar Z,\bar  t)\in \Delta_i^l.
\end{eqnarray}
In the last conclusion we have also used Lemma \ref{T:doubling}. Using these facts, and using  the definition of the good $\epsilon_0$ cover, we see that
\begin{equation}\label{yy8}
\begin{split}
|II_2|&\leq\frac {c_{\eta_2}}{\omega(\Delta_i^l)}\sum_{j=l+2}^k\omega((\tilde{\mathcal{O}}_{j-1}\setminus\mathcal{O}_j)\cap \Delta_i^l)\\
&\leq\frac {c_{\eta_2}}{\omega(\Delta_i^l)}\sum_{j=l+2}^k\omega(\mathcal{O}_{j-1}\cap \Delta_i^l)\leq \frac {c_{\eta_2}}{ \omega(\Delta_i^l)}\sum_{j=l+2}^k\epsilon_0^{j-1-l}\omega(\Delta_i^l)\leq C_{\eta_2}\epsilon_0.
\end{split}
\end{equation}
To estimate the term $II_3$ we first observe that $\hat \Delta_i^l\cap \tilde{\mathcal{O}}_{l}=\emptyset$ by the definition of $\tilde{\mathcal{O}}_{l}$. Hence,
\begin{eqnarray}
II_3&=&\omega( P^-,\hat\Delta_i^l\cap (\tilde{\mathcal{O}}_{j-1}\setminus\mathcal{O}_j))=0
\end{eqnarray}
and we can conclude that
\begin{eqnarray}\label{yy11}
u( P^-)\leq c\eta_2^\sigma+C_{\eta_2}\epsilon_0.
\end{eqnarray}
Combining \eqref{yy10} and \eqref{yy11} we can conclude, in either case, that
\begin{eqnarray}
u( P^+)-u( P^-)\geq c_{\eta_1}^{-1}-C_{\eta_1}\epsilon_0-c\eta_2^\sigma-C_{\eta_2}\epsilon_0.
\end{eqnarray}
We now first choose $\eta_1=\eta_1(m,\kappa,M_1)$ small. We then choose $\eta_2=\eta_2(m,\kappa,M_1)$ so that $c_{\eta_1}^{-1}=2c\eta_2^\sigma$. Having fixed $\eta_1$ and $\eta_2$ we
choose $\epsilon_0=\epsilon_0(m,\kappa,M_1)$ so that $c\eta_2^\sigma=2(C_{\eta_1}+C_{\eta_2})\epsilon_0$. By these choices we can conclude that there exists $0<\beta=\beta(m,\kappa,M_1)\ll 1$ such that
\begin{eqnarray}\label{yy12}
u(P+d_{\Delta_i^l}e_m)-u( P+l(\bar\Delta_i^l)e_m)\geq \beta,\ \forall P\in\{(x,\psi(x,y,t),y,y_m,t)\in \bar\Delta_i^l\}.
\end{eqnarray}
In particular, fix $P\in\{(x,\psi(x,y,t),y,y_m,t)\in \bar\Delta_i^l\}$. Then \eqref{yy12} implies
\begin{eqnarray}
\beta^2\leq cl(\Delta_i^l)\int_{P^-}^{P^+}|\partial_{x_m}u(x,x_m,y,y_m,t)|^2\, \d x_m.
\end{eqnarray}
Integrating with respect to $P\in  \bar\Delta_i^l$ we see that
\begin{eqnarray}\label{yy14}
\beta^2\sigma(\bar\Delta_i^l)\leq cl(\Delta_i^l)\iiint_{R_{\bar\Delta_i^l}}|\nabla_Xu(Z,t)|^2\, \d Z\d t,
\end{eqnarray}
where $R_{\bar\Delta_i^l}$ is a naturally defined Whitney type region. Recall that $\sigma(\bar\Delta_i^l)\approx \sigma(\Delta_i^l)$. In particular, by an elementary connectivity/covering argument we see that
\begin{eqnarray*}
\quad c^{-1}\beta^2\leq \iiint_{\tilde W_{\Delta_i^l}} |\nabla_Xu|^2\delta^{2-{\bf q}}\, dZdt,
\end{eqnarray*}
where $\tilde W_{\Delta_i^l}$ is a natural Whitney type region associated to $\Delta_i^l$, $\delta=\delta(Z,t)$ is the distance from $(Z,t)$ to $\Sigma$,  and $c=c(m,M_1,\kappa)$, $1\leq c<\infty$. Consequently, for
 $(Z_0,t_0)\in E$ fixed we find, by summing over all indices $i$, $l$, such that
$(Z_0,t_0)\in \Delta_i^l$, that
\begin{eqnarray}
\quad\quad c^{-1}\beta^2k\leq\sum_{i,l: (Z_0,t_0)\in \Delta_i^l}  \biggl (\iiint_{\tilde W_{\Delta_i^l}} |\nabla_Xu|^2\delta^{2-{\bf q}}\, \d Z\d t\biggr ).
\end{eqnarray}
The construction can be made so that the Whitney type regions $\{\tilde W_{\Delta_i^l}\}$  have bounded overlaps measured by a constant depending only on $m$, $M_1$, and such that $W_{\Delta_i^l}\subset T_{cQ_0}$ for some $c=c(m,M_1)$, $1\leq c<\infty$, where $T_{cQ_0}$ is defined in \eqref{eq2.box}. Hence, integrating with respect to $\d \sigma$, we deduce that
\begin{eqnarray}\label{yy13+}
\quad c^{-1}\beta^2k\sigma(E)\leq \biggl (\iiint_{ T_{cQ_0}} \bigl |\nabla_Xu|^2\delta\, \d Z\d t\biggr )
\end{eqnarray}
where, resolving the dependencies, $c=c(m,\kappa,M_1)$, $1\leq c<\infty$. Furthermore,
$$k\approx \frac {\log(\delta_0)}{\log (\epsilon_0)},$$
where $\eta$ and $\epsilon_0$ now have been fixed, and $\delta_0$ is at our disposal. Given $\Upsilon$
we obtain the conclusion of the lemma by specifying $\delta_0=\delta_0(m,\kappa,M_1,\Upsilon)$  sufficiently small. This completes the proof of Lemma \ref{lemmacruc}.
\qed

\end{document}